\documentclass[11pt,a4paper]{amsart}
\usepackage[T1]{fontenc}
\usepackage[utf8]{inputenc}
\usepackage{lmodern}
\usepackage[margin=27mm]{geometry}
\usepackage{amsmath,amssymb,amsthm,mathtools,mathrsfs}
\usepackage{microtype}
\usepackage{graphicx}
\usepackage[hidelinks]{hyperref}

\numberwithin{equation}{section}
\newtheorem{theorem}{Theorem}[section]
\newtheorem{proposition}[theorem]{Proposition}
\newtheorem{lemma}[theorem]{Lemma}

\theoremstyle{remark}
\newtheorem{remark}[theorem]{Remark}
\newcommand{\R}{\mathbb R}
\newcommand{\C}{\mathbb C}
\DeclareMathOperator{\sech}{sech}

\title[Stability of the GP branch in a coupled system]
{Global well-posedness and orbital stability of the
Gross--Pitaevskii branch in a coupled defocusing Schr\"odinger system}
\author{Micha{\l} Kowalczyk}
\address{Departamento de Ingenier\'{\i}a Matem\'atica and CMM,
Universidad de Chile, Casilla 170 Correo 3, Santiago, Chile}
\email{kowalczy@dim.uchile.cl}
\author{Jaros{\l}aw Mederski}
\address{Institute of Mathematics, Polish Academy of Sciences,
ul.\ \'Sniadeckich 8, 00-656 Warsaw, Poland}
\address{and Faculty of Mathematics and Computer Science,
Nicolaus Copernicus University in Toru\'n,
ul.\ Chopina 12/18, 87-100 Toru\'n, Poland}
\email{jmederski@impan.pl}
\author{Isidora Miranda}
\address{Departamento de Ingenier\'{\i}a Matem\'atica and CMM,
Universidad de Chile, Casilla 170 Correo 3, Santiago, Chile}
\email{isimimi311@gmail.com}
\thanks{The first author was supported by the Chilean grants FONDECYT
1250156 and 1210405 and ANID FB 210005.
Part of this work was done during the second author's visit to Chile
funded by the FONDECYT International Cooperation grant 1210405,
and he gratefully acknowledges the support and hospitality received
during his visit. He was also partly supported by the National Science Centre, Poland
(Grant No.\ 2025/57/B/ST1/04804). The third author was partly supported by the Chilean
grant FONDECYT 1250156.}
\subjclass[2020]{35Q55, 35B35, 35C08}
\keywords{Coupled Gross--Pitaevskii equations, nonzero boundary conditions,
global well-posedness, orbital stability, dark soliton, conserved momentum}
\date{}

\begin{document}
\begin{abstract}
We consider a one-dimensional coupled defocusing Schr\"odinger system
with a nonzero background in its first component and finite mass in
its second component. We prove global well-posedness in the full energy
space, together with conservation of energy, bright mass, and total
untwisted momentum. For all nonnegative coupling coefficients, every
subsonic Gross--Pitaevskii soliton with vanishing second component is orbitally
stable against perturbations of both components, including the
zero-speed black soliton. At nonzero speed the proof uses a fixed
energy--momentum functional augmented by the conserved bright mass.
At zero speed it uses weighted coercivity and a local real lift of the
momentum. We also identify the sharp threshold for positivity of the
unaugmented bright energy and explain why this threshold is not a
necessary condition for orbital stability.
\end{abstract}
\maketitle

\section{Introduction and main result}\label{sec:introduction}

We study
\begin{equation}\label{eq:system}
\begin{aligned}
i\psi_t&=\psi_{xx}+\psi(1-|\psi|^2-\alpha|\phi|^2),\\
i\phi_t&=\phi_{xx}+\phi(1-\alpha|\psi|^2-\beta|\phi|^2),
\qquad (t,x)\in\R\times\R.
\end{aligned}
\end{equation}
For global well-posedness and stability of the GP branch we assume
\begin{equation}\label{eq:parameters}
\alpha\ge0,\qquad \beta\ge0.
\end{equation}
No positive-definiteness assumption on the interaction matrix is needed
for these results. The explicit dark--bright family in
Section~\ref{sec:dark-bright-open} will be considered under the stronger
conditions $\alpha>1$ and $\beta>\alpha^2$, stated separately there.
We refer to
\eqref{eq:system} as a coupled Gross--Pitaevskii or Manakov-type system;
no integrable structure is assumed.

For every $|c|<\sqrt2$, the pair $(U_c(x-ct),0)$ is a solution, where
\begin{equation*}
U_c(x)=R_c(x)+iI_c
 =\gamma_c\tanh\!\left(\frac{\gamma_c x}{\sqrt2}\right)
      -\frac{ic}{\sqrt2},
\qquad \gamma_c=\sqrt{1-\frac{c^2}{2}}.
\end{equation*}
The case $c=0$ is the black soliton. The second component of this
branch is identically zero, but the perturbations considered below
are not required to have this property.

\subsection{Energy space and distances}
Set
\[
\mathcal E=\{u\in H^1_{\mathrm{loc}}(\R;\C):
           u_x\in L^2(\R),\ 1-|u|^2\in L^2(\R)\},
\qquad
\mathcal X=\mathcal E\times H^1(\R;\C).
\]
The Cauchy theory will use the complete metric
\begin{equation}\label{eq:energy-metric}
\begin{aligned}
d_{\mathcal X}((\psi,\phi),(\widetilde\psi,\widetilde\phi))
={}&\|\psi-\widetilde\psi\|_\infty
  +\|\psi_x-\widetilde\psi_x\|_2\\
&+\big\||\psi|^2-|\widetilde\psi|^2\big\|_2
  +\|\phi-\widetilde\phi\|_{H^1}.
\end{aligned}
\end{equation}
Every $u\in\mathcal E$ is bounded and satisfies $|u(x)|\to1$ as
$|x|\to\infty$; these facts are proved in Section~\ref{sec:gwp}.
We do not assume that its phase has limits at infinity.

The conserved energy and bright mass are
\begin{align}
H(\psi,\phi)
={}&\frac12\int(|\psi_x|^2+|\phi_x|^2)
    +\frac14\int(1-|\psi|^2)^2
    +\frac12\int(\alpha|\psi|^2-1)|\phi|^2
    +\frac\beta4\int|\phi|^4,\label{eq:energy}\\
M(\phi)={}&\frac12\int|\phi|^2.\label{eq:mass}
\end{align}
All integrals without an indicated domain are over $\R$.
Theorem~\ref{thm:gwp} proves global well-posedness and conservation
of $H$ and $M$. Proposition~\ref{prop:momentum} treats the total
untwisted momentum, including dark components with zeros.

For the orbital statement define
\begin{equation}\label{eq:weighted-norm}
Q_c=1-|U_c|^2
 =\gamma_c^2\sech^2\!\left(\frac{\gamma_c x}{\sqrt2}\right)
 =\sqrt2 R_c',
\qquad
\|\epsilon\|_{H_c^1}^2=\|\epsilon_x\|_2^2+\int Q_c|\epsilon|^2.
\end{equation}
Here the weighted Hilbert space is
\[
 H_c^1=\bigl\{f\in H^1_{\mathrm{loc}}(\R;\C):
       f_x\in L^2(\R),\ \sqrt{Q_c}\,f\in L^2(\R)\bigr\},
\]
equipped with the norm in \eqref{eq:weighted-norm}. Since $Q_c$ has
a positive lower bound on each compact interval, a Cauchy sequence
in this norm converges locally in $H^1$; its derivative and weighted
$L^2$ limits identify the limit in $H_c^1$, proving completeness.
Define the centered and orbital distances by
\begin{align}
\mathfrak d_c(\psi,\phi)
 &=\|\psi-U_c\|_{H_c^1}
      +\big\||\psi|-|U_c|\big\|_2+\|\phi\|_{H^1},\nonumber\\
d_c^{\mathrm{orb}}(\psi,\phi)
 &=\inf_{a,\theta\in\R}
   \mathfrak d_c(e^{-i\theta}\psi(\cdot+a),\phi(\cdot+a)).
   \label{eq:orbit-distance}
\end{align}
The weighted norm is evaluated after centering the solution at the
profile. In particular, \eqref{eq:orbit-distance} is not an assertion
of stability in the global $L^\infty$ phase metric
\eqref{eq:energy-metric}.

\paragraph{\bf{Comparison of the flow and orbital distances.}}
For each fixed $|c|<\sqrt2$, the flow metric is strictly stronger
than the centered distance used in the stability statement. More
precisely, there is $C_c>0$ such that, for every $(\psi,\phi)\in\mathcal X$,
\[
 d_c^{\mathrm{orb}}(\psi,\phi)
 \le\mathfrak d_c(\psi,\phi)
 \le C_c\,d_{\mathcal X}\bigl((\psi,\phi),(U_c,0)\bigr).
\]
This estimate also holds at $c=0$. Indeed, with $\epsilon=\psi-U_c$,
\[
 \|\epsilon\|_{H_c^1}
 \le\|\epsilon_x\|_2+\|Q_c\|_1^{1/2}\|\epsilon\|_\infty.
\]
Choose $R,m>0$ such that $|U_c(x)|\ge m$ for $|x|\ge R$.
On $[-R,R]$ use the reverse triangle inequality; on its complement
factor the difference of squared moduli. This gives
\[
 \big\||\psi|-|U_c|\big\|_2
 \le\sqrt{2R}\|\psi-U_c\|_\infty
     +m^{-1}\big\||\psi|^2-|U_c|^2\big\|_2,
\]
which proves the comparison without a smallness assumption.

The converse estimate fails because the weighted distance permits a
slowly varying phase far from the soliton core. To see this explicitly,
choose a smooth real function $\chi$ with $\chi(s)=0$ for $s\le1$
and $\chi(s)=\pi$ for $s\ge2$, and set
\[
 \psi_n(x)=e^{i\chi(x/n)}U_c(x),\qquad \phi_n=0,\qquad n\ge1.
\]
These functions belong to $\mathcal X$ and satisfy
$|\psi_n|=|U_c|$. Moreover,
\[
\begin{aligned}
 \|(\psi_n-U_c)_x\|_2
 &\le2\|U_c'\|_{L^2((n,\infty))}
       +n^{-1/2}\|\chi'\|_2\longrightarrow0,\\
 \int Q_c|\psi_n-U_c|^2
 &\le4\int_n^\infty Q_c\longrightarrow0.
\end{aligned}
\]
Here $|U_c|\le1$ was used in the derivative estimate. Consequently
\[
 \mathfrak d_c(\psi_n,0)\longrightarrow0,
 \qquad
 d_{\mathcal X}\bigl((\psi_n,0),(U_c,0)\bigr)
 \ge\|\psi_n-U_c\|_\infty=2.
\]
Even optimizing the stronger distance over translations and a
constant phase does not restore convergence. For every $a,\theta\in\R$,
the limiting discrepancies at the two ends of the line give
\[
 \|e^{-i\theta}\psi_n(\cdot+a)-U_c\|_\infty
 \ge\max\{|e^{-i\theta}-1|,\ |e^{-i\theta}+1|\}
 \ge\sqrt2.
\]

Thus global well-posedness in $d_{\mathcal X}$ and orbital stability
in $d_c^{\mathrm{orb}}$ are distinct but compatible assertions.
The former provides the global solution and its strong time
continuity; the latter controls its centered energy perturbation,
without asserting global uniform control of the complex phase.
Continuous dependence on initial data in the weaker topology does
not follow merely from continuous dependence in $d_{\mathcal X}$.
Finally, the quantity $D_c$ in
Proposition~\ref{prop:augmented-coercivity} uses the density defect
$\||\psi|^2-|U_c|^2\|_2$ instead of the modulus defect.
Lemma~\ref{coer:metric} shows that the complete distances $D_c$ and
$\mathfrak d_c$ are locally equivalent, also at $c=0$; this does
not assert equivalence of the two defect terms by themselves.

\begin{theorem}[Orbital stability of the GP branch]\label{thm:stability}
Assume \eqref{eq:parameters} and fix $c\in(-\sqrt2,\sqrt2)$.
There exist $\delta_c,C_c>0$, depending also on $\alpha,\beta$, such
that every $(\psi_0,\phi_0)\in\mathcal X$ with
$\delta=d_c^{\mathrm{orb}}(\psi_0,\phi_0)<\delta_c$ gives a global
solution satisfying
\begin{equation*}
\sup_{t\in\R}d_c^{\mathrm{orb}}(\psi(t),\phi(t))\le C_c\delta.
\end{equation*}
For $c\ne0$, this solution also satisfies
$\inf_{t,x}|\psi(t,x)|\ge |c|/(2\sqrt2)$.
No estimate uniform up to the sonic endpoints is asserted.
\end{theorem}

The proof separates $c\ne0$ from $c=0$. At nonzero speed we construct
a coercive conserved functional from energy, total momentum, and
bright mass. At zero speed we retain the exact variation of
$H(U_{c(t)},0)$ along the modulated branch and absorb it using the
momentum slope near $c=0$. Proposition~\ref{coer:threshold} identifies
the sharp threshold for coercivity of the unaugmented energy on the
constraint slice; this threshold is not required for orbital stability.

\subsection{Relation to previous work}
The scalar one-dimensional Cauchy theory with nonvanishing data is
developed by Gallo \cite{Gallo2004,Gallo2008}. The extended linear
flow and moving-background construction used here are closely
related to G\'erard's energy-space approach
\cite[Lemma~3, Proposition~2.3 and Remark~3.2]{Gerard2006}.
His main global theorem concerns dimensions two and three; we prove
the required one-dimensional coupled result directly.

Scalar orbital stability at nonzero speed is part of the
energy--momentum theory developed by Lin \cite{Lin2002} and
B\'ethuel--Gravejat--Smets
\cite[Proposition~4 and Section~2.1]{BGS2014};
see also \cite[Theorem~1]{BGS2015}.
The scalar Lyapunov functional with a squared momentum correction
and its quantitative stability estimate are developed by Chiron
\cite[Theorem~12 and Appendix~A]{Chiron2013}.
Black-soliton orbital stability was established by
B\'ethuel--Gravejat--Saut--Smets \cite{BGSS2008}.
Our zero-speed argument uses the weighted framework and local
momentum expansion of Gravejat--Smets
\cite[Propositions~4--5]{GS2015}.
We do not claim these scalar results or the Cauchy-problem techniques
as new. The nonzero-speed scalar coercivity is reproduced below
in amplitude variables, and the coupled conservation laws and
bright-mass augmentation are verified explicitly.

L\'opez-Mart\'inez \cite{Lopez2026} constructs dark--bright waves
variationally using prescribed momentum and bright mass.
The present paper instead studies stability of the branch
$(U_c,0)$ against two-component perturbations. It does not establish
the stability of travelling waves with a nonzero bright profile.

\begin{remark}[Open problem: stability of nontrivial dark--bright waves]
\label{rem:dark-bright-open}
Under the additional assumptions $\alpha>1$ and $\beta>\alpha^2$,
consider travelling waves, allowing a
uniform phase rotation of the bright component, of the form
\[
 (\psi(t,x),\phi(t,x))
   =\bigl(U(x-ct),e^{i\nu t}V(x-ct)\bigr),
 \qquad (U,V)\in\mathcal X,\quad V\not\equiv0.
\]
Determine which such waves are orbitally stable, and which are
unstable, modulo translations and independent constant phase rotations
of the two components. Theorem~\ref{thm:stability} concerns only
the orbit of $(U_c,0)$ and does not settle this question.
Section~\ref{sec:dark-bright-open} gives an explicit family
$(U_{d,c},V_{b,c})$, discusses its conserved action, and formulates
the corresponding open problem in a precise energy topology.
\end{remark}

This paper is organized as follows. Section~\ref{sec:gwp} gives the Cauchy theory.
Section~\ref{sec:coercivity} proves the weighted coercivity estimates
and distinguishes them from positivity of the unaugmented energy.
Section~\ref{sec:momentum} establishes momentum conservation.
Sections~\ref{sec:nonzero} and \ref{sec:black} prove
Theorem~\ref{thm:stability}.
Section~\ref{sec:dark-bright-open} discusses the nontrivial dark--bright
family and the remaining stability problem.

\section{Global well-posedness}\label{sec:gwp}

We first construct the flow on the full energy space. The argument follows
the scalar nonzero-background Cauchy theory of Gallo
\cite{Gallo2004,Gallo2008}; see also G\'erard \cite{Gerard2006} for the
scalar energy-space approach. We give the coupled argument explicitly:
the first component has a nonzero background, whereas the second belongs to
$H^1(\R)$. Only the one-dimensional Sobolev algebra property is needed for
the local construction. Throughout this section we assume
\eqref{eq:parameters}, as in the stability theorem.

For $U=(\psi,\phi)\in\mathcal X$, write
\begin{equation}\label{gwp:size}
 \Lambda(U)=\|\psi\|_\infty+\|\psi_x\|_2
       +\|1-|\psi|^2\|_2+\|\phi\|_{H^1}.
\end{equation}
The first term is finite. Indeed, every $u\in\mathcal E$ satisfies
\begin{equation}\label{gwp:embedding}
 \|u\|_\infty\le C\left(1+
       \|u_x\|_2^{1/3}\|1-|u|^2\|_2^{1/3}\right).
\end{equation}
To see this without assuming boundedness beforehand, put $A=\|u_x\|_2$.
If $A=0$, the energy condition makes $u$ a constant of modulus one.
Otherwise, for any $x_0\in\R$ and $m=|u(x_0)|$, the estimate
$|u(x)-u(x_0)|\le A|x-x_0|^{1/2}$ gives $|u(x)|\ge m/2$ on
$|x-x_0|\le m^2/(4A^2)$. For $m\ge\sqrt8$, integration on that interval
gives $\|1-|u|^2\|_2^2\ge m^6/(128A^2)$. Taking the supremum proves
\eqref{gwp:embedding}. Moreover, $q=1-|u|^2$ belongs to $H^1$,
since $q_x=-2\operatorname{Re}(\bar u u_x)\in L^2$.
Thus $q(x)\to0$ as $|x|\to\infty$, and $|u(x)|\to1$.
In particular, the functionals in
\eqref{eq:energy} and \eqref{eq:mass} are finite on $\mathcal X$ and
continuous in $d_{\mathcal X}$.

We shall use two elementary properties of this space. First,
\begin{equation}\label{gwp:affine-inclusion}
 \mathcal E+H^1(\R)\subset\mathcal E,
\end{equation}
as follows by expanding $|u+h|^2$ and using
$H^1(\R)\subset L^\infty(\R)$. Second,
$(\mathcal X,d_{\mathcal X})$ is complete. A Cauchy sequence has uniformly
convergent dark components, $L^2$-convergent derivatives and density
defects, and $H^1$-convergent bright components. Distributional
differentiation and local uniform convergence identify the first three
limits as $\psi$, $\psi_x$, and $1-|\psi|^2$, respectively.

Let $S(t)=e^{-it\partial_x^2}$ be the free Schr\"odinger group. It is
convenient to extend it first to
\[
 X^1=\{u\in L^\infty(\R):u_x\in L^2(\R)\}.
\]
For $u\in X^1$, define
\begin{equation}\label{gwp:extended-group}
 S_E(t)u=u+A_tu,\qquad
 \widehat{A_tu}(\xi)=
       \frac{e^{it\xi^2}-1}{i\xi}\widehat{u_x}(\xi),
\end{equation}
where the multiplier is set equal to zero at $\xi=0$.

\begin{theorem}[Global energy-space flow]\label{thm:gwp}
Suppose that $\alpha,\beta\ge0$. For every
$U_0=(\psi_0,\phi_0)\in\mathcal X$, system \eqref{eq:system} has a
unique global mild solution $U\in C(\R;\mathcal X)$ satisfying
\begin{equation}\label{gwp:solution-class}
 \psi-S_E(\cdot)\psi_0\in C(\R;H^1(\R)),\qquad
 \phi-S(\cdot)\phi_0\in C(\R;H^1(\R)).
\end{equation}
For every $T>0$, the solution map is continuous from
$(\mathcal X,d_{\mathcal X})$ to $C([-T,T];\mathcal X)$, with the
uniform metric. Moreover,
\begin{equation*}
 H(U(t))=H(U_0),\qquad M(\phi(t))=M(\phi_0)
 \quad(t\in\R).
\end{equation*}
No nonvanishing condition or phase limit at spatial infinity is imposed
on $\psi_0$.
\end{theorem}

\begin{proof}
\emph{The linear flow and affine charts.}
The inequality
\[
 \frac{|e^{it\xi^2}-1|}{|\xi|}
 \le\min\{|t||\xi|,2|\xi|^{-1}\}\le\sqrt{2|t|}
\]
and unitarity on $L^2$ imply
\begin{equation}\label{gwp:linear-increment}
 \|A_tu\|_2\le\sqrt{2|t|}\|u_x\|_2,\qquad
 \|(A_tu)_x\|_2\le2\|u_x\|_2,
 \qquad \|A_tu\|_{H^1}\longrightarrow0\quad(t\to0).
\end{equation}
For the last assertion use
$(A_tu)_x=(S(t)-I)u_x$ and strong continuity on $L^2$.
Thus $S_E(t)$ maps $X^1$ into itself and, by
\eqref{gwp:affine-inclusion}, maps $\mathcal E$ into itself.
It agrees with $S(t)$ on the subspace $H^1\subset X^1$.
The multiplier identity
$A_{t+s}u=A_tu+S(t)A_su$ proves the group law.
The same definition gives the free equation in distributions.

Let $d_E$ denote the sum of the three dark-component terms in
$d_{\mathcal X}$ and put
$\Lambda_E(u)=\|u\|_\infty+\|u_x\|_2+\|1-|u|^2\|_2$.
Expansion of squared moduli and \eqref{gwp:linear-increment} give
\begin{equation}\label{gwp:chart-estimates}
 \begin{aligned}
 d_E(S_E(t)u,S_E(t)v)&\le C_{R,T}d_E(u,v),\qquad |t|\le T,\\
 d_E(u+h,v+k)&\le C_R\bigl(d_E(u,v)+\|h-k\|_{H^1}\bigr).
 \end{aligned}
\end{equation}
In the first estimate assume $\Lambda_E(u)+\Lambda_E(v)\le R$;
in the second assume
$\|u\|_\infty+\|v\|_\infty+\|h\|_{H^1}+\|k\|_{H^1}\le R$.
For example, the density difference in the second estimate is
\[
 |u+h|^2-|v+k|^2=|u|^2-|v|^2
 +2\operatorname{Re}\bigl(\overline u(h-k)
                    +(\overline u-\overline v)k\bigr)
 +|h|^2-|k|^2.
\]
These estimates also show that $S_E(t)$ is strongly continuous in $d_E$.

\emph{Local existence and continuation.}
Write the nonlinearity as
\[
 N(\psi,\phi)=
 \begin{pmatrix}
 (1-|\psi|^2-\alpha|\phi|^2)\psi\\
 (1-\alpha|\psi|^2-\beta|\phi|^2)\phi
 \end{pmatrix}.
\]
It maps $\mathcal X$ into $H^1\times H^1$ and satisfies
\begin{equation}\label{gwp:nonlinear-estimate}
 \|N(U)\|_{H^1\times H^1}\le C_R,\qquad
 \|N(U)-N(V)\|_{H^1\times H^1}
       \le C_Rd_{\mathcal X}(U,V)
\end{equation}
when $\Lambda(U),\Lambda(V)\le R$. Here and below constants may depend
on the fixed coefficients $\alpha,\beta$. To verify these estimates,
use the product inequality
\[
 \|fg\|_{H^1}\le C(\|f\|_\infty+\|f_x\|_2)\|g\|_{H^1},
 \qquad f\in X^1,\quad g\in H^1,
\]
and
\[
 \begin{aligned}
 \bigl\||\psi|^2-|\widetilde\psi|^2\bigr\|_{H^1}
 \le{}&\bigl\||\psi|^2-|\widetilde\psi|^2\bigr\|_2\\
 &+C_R\bigl(\|\psi-\widetilde\psi\|_\infty
                 +\|\psi_x-\widetilde\psi_x\|_2\bigr).
 \end{aligned}
\]
In particular, $1-|\psi|^2\in H^1$. The expression
$(1-|\psi|^2)\psi$ must be kept grouped: its two expanded terms need
not separately belong to $L^2$.

Put $U_L(t)=(S_E(t)\psi_0,S(t)\phi_0)$ and seek $U=U_L+W$.
On the closed unit ball of
$Y_T=C([-T,T];H^1\times H^1)$ define
\begin{equation*}
 \mathcal T W(t)=-i\int_0^t
 \begin{pmatrix}S(t-s)&0\\0&S(t-s)\end{pmatrix}
 N(U_L(s)+W(s))\,ds.
\end{equation*}
If $\Lambda(U_0)\le R$, then \eqref{gwp:linear-increment} and
\eqref{gwp:chart-estimates} bound the intrinsic size of $U_L+W$
uniformly for $|t|\le1$. Unitarity and
\eqref{gwp:nonlinear-estimate} therefore give
\[
 \|\mathcal TW\|_{Y_T}\le TC_R,\qquad
 \|\mathcal TW-\mathcal T\widetilde W\|_{Y_T}
       \le TC_R\|W-\widetilde W\|_{Y_T}.
\]
Choose $T=T(R)>0$ with $T\le1$ and $TC_R\le1/2$.
The contraction theorem supplies a fixed point, and
\eqref{gwp:affine-inclusion} gives $U\in C([-T,T];\mathcal X)$.
The fixed-point equation is precisely the mild formulation of
\eqref{eq:system}. Every other mild solution in
\eqref{gwp:solution-class} has a correction vanishing at the initial
time and hence lies in the same unit ball on a shorter interval.
Iteration gives uniqueness on every common interval of existence.

For two initial data of bounded intrinsic size,
\eqref{gwp:chart-estimates} and \eqref{gwp:nonlinear-estimate} similarly
give
\[
 \|W-\widetilde W\|_{Y_T}
 \le TC_R\|W-\widetilde W\|_{Y_T}
             +C_Rd_{\mathcal X}(U_0,\widetilde U_0).
\]
Absorbing the first term proves local continuous dependence in the full
metric, not only for data in one affine $H^1$ chart.
Restarting the construction yields a maximal interval $(-T_-,T_+)$ and
the continuation alternative
\begin{equation}\label{gwp:continuation}
 T_+<\infty\quad\Longrightarrow\quad
 \limsup_{t\uparrow T_+}\Lambda(U(t))=\infty,
\end{equation}
and its negative-time counterpart. Indeed, a uniform bound permits restarting with the same lifespan
$T(R)$ arbitrarily close to an endpoint.

\emph{Regular solutions and conservation.}
We justify conservation before using it for rough solutions.
For initial data with $\psi_{0x}\in H^2$ and $\phi_0\in H^3$, the same
construction applies in affine $H^3\times H^3$ charts. The free
increments belong to $H^3$, and the product rule makes $N$ locally
Lipschitz in these charts. More generally, the local construction at
order $m=2,3$ has a lifespan depending on $\Lambda+R_m$, where
\[
 R_m=\|\psi_x\|_{H^{m-1}}+\|\phi\|_{H^m}.
\]
Uniqueness in the energy class identifies these regular solutions with
the previously constructed solutions.

For completeness, the requisite persistence estimates at orders two
and three do not require a bound on $\psi_x$ in $L^\infty$ from the
energy alone. On a set where $\Lambda\le R$, the product rule gives
\begin{equation*}
 \begin{aligned}
 \|N(U)\|_{H^2\times H^2}&\le C_R(1+R_2),\\
 \|N(U)\|_{H^3\times H^3}&\le C_{R,R_2}(1+R_3).
 \end{aligned}
\end{equation*}
At order two, products of first derivatives are controlled by
\[
 \|v\|_4^2\le C\|v\|_2^{3/2}\|v_x\|_2^{1/2},
 \qquad v=\psi_x\ \hbox{or}\ \phi_x.
\]
The undifferentiated density defect is already bounded in $L^2$.
At order three, first derivatives are bounded in $L^\infty$ by $R_2$;
terms containing a second derivative and a first derivative are treated
in $L^2\cdot L^\infty$, and the highest derivative occurs linearly.
Applying the unitary estimates to the differentiated Duhamel equations
and then Gronwall's inequality proves that $R_2$, followed by $R_3$,
remains bounded on every finite interval where $\Lambda$ is bounded.

For these regular solutions, $\psi_t,\phi_t\in H^1$.
The energy is continuously differentiable on the affine
$H^1\times H^1$ chart, and its real gradients are
\[
 F_\psi=-\psi_{xx}-(1-|\psi|^2-\alpha|\phi|^2)\psi,
 \qquad
 F_\phi=-\phi_{xx}-(1-\alpha|\psi|^2-\beta|\phi|^2)\phi.
\]
Both belong to $L^2$; system \eqref{eq:system} gives
$\psi_t=iF_\psi$ and $\phi_t=iF_\phi$. The chain rule and weak
integration by parts thus yield
\[
 \frac{dH}{dt}=\operatorname{Re}\int
       (F_\psi\overline{\psi_t}+F_\phi\overline{\phi_t})=0,
 \qquad
 \frac{dM}{dt}=\operatorname{Re}\int\phi_t\overline\phi=0.
\]
All pairings are integrable. Equivalently, one may insert compact
cutoffs before integrating by parts; the derivative errors are bounded
by $CR^{-1}\|\psi_x\|_2\|\psi_t\|_2$ and the analogous bright term,
and vanish as the cutoff radius $R$ tends to infinity.

The useful nonnegative conserved combination is
\begin{equation}\label{gwp:positive-energy}
 \begin{aligned}
 \mathcal K(U):=H(U)+M(\phi)
 ={}&\frac12\int(|\psi_x|^2+|\phi_x|^2)
       +\frac14\int(1-|\psi|^2)^2\\
    &+\frac\alpha2\int|\psi|^2|\phi|^2
       +\frac\beta4\int|\phi|^4.
 \end{aligned}
\end{equation}
Writing $\mathcal K_0=\mathcal K(U_0)$ and $M_0=M(\phi_0)$, conservation gives
\begin{equation}\label{gwp:energy-bounds}
 \|\psi_x(t)\|_2^2+\|\phi_x(t)\|_2^2\le2\mathcal K_0,
 \qquad \|1-|\psi(t)|^2\|_2\le2\sqrt{\mathcal K_0},
 \qquad \|\phi(t)\|_2^2=2M_0.
\end{equation}
Together with \eqref{gwp:embedding}, these bounds control $\Lambda$
uniformly on the regular lifespan. The persistence estimates above and
the regular local theory then show that regular solutions are global.

\emph{Approximation and global continuation for arbitrary data.}
Let $\rho_\varepsilon$ be a standard mollifier. Although $u\in\mathcal E$
need not be in $L^2$, the fundamental theorem of calculus gives
\[
 \|u(\cdot-y)-u\|_2\le |y|\|u_x\|_2,
 \qquad \|u(\cdot-y)-u\|_\infty\le |y|^{1/2}\|u_x\|_2.
\]
Consequently
\begin{equation*}
 \|\rho_\varepsilon*u-u\|_2\le C\varepsilon\|u_x\|_2,
 \qquad
 \|\rho_\varepsilon*u-u\|_\infty
                    \le C\varepsilon^{1/2}\|u_x\|_2.
\end{equation*}
The derivatives converge in $L^2$, and
\[
 \bigl\||\rho_\varepsilon*u|^2-|u|^2\bigr\|_2
 \le(\|\rho_\varepsilon*u\|_\infty+\|u\|_\infty)
                            \|\rho_\varepsilon*u-u\|_2\longrightarrow0.
\]
Mollification of both components therefore approximates every datum in
$d_{\mathcal X}$ by data having the regularity just used.

Fix a compact interval $I$ inside the maximal lifespan of the rough
solution. Iteration of the local dependence estimate on finitely many
common intervals shows convergence of the corresponding regular
solutions in $C(I;\mathcal X)$. Continuity of $H$ and $M$ transfers their
conservation to the rough solution on $I$. Thus conservation is proved
on its maximal interval before global existence is inferred.
The bounds \eqref{gwp:energy-bounds} and \eqref{gwp:embedding} now apply
there and exclude a finite endpoint in \eqref{gwp:continuation}.
Finally, nearby initial data have uniformly bounded $\mathcal K_0$ and $M_0$.
Iteration of the local dependence estimate therefore gives continuity
of the global flow on every fixed compact time interval.
\end{proof}

\begin{remark}\label{gwp:scope}
The local theory above allows arbitrary real $\alpha,\beta$.
Their nonnegativity is used only in \eqref{gwp:positive-energy}.
The proof does not assert that energy controls the $L^2$ norm of a dark
perturbation relative to a fixed background. Instead, the local flow
can be restarted using the intrinsic size \eqref{gwp:size}, controlled
by energy and bright mass. Conservation of total untwisted momentum
is established separately below and is not used to construct this flow.
\end{remark}

\section{Constrained energy coercivity}\label{sec:coercivity}

In this section \eqref{eq:parameters} is assumed. Write
$\gamma=\gamma_c$ and $U_c=R_c+iI_c$, and set
\[
\eta=|U_c+\epsilon|^2-|U_c|^2
    =2(R_c\epsilon_1+I_c\epsilon_2)+|\epsilon|^2,
\qquad \zeta=|\omega|^2.
\]
The real energy gradient at $(U_c,0)$ is $(icR_c',0)$.
Consequently, direct expansion gives
\begin{equation}\label{coer:expansion}
\begin{aligned}
H(U_c+\epsilon,\omega)-H(U_c,0)
={}&c\int R_c'\epsilon_2+\frac12q_c(\epsilon_1)
       +\frac12q_c(\epsilon_2)+\frac12m_c(\omega)\\
 &+\frac14\int\eta^2+\frac\alpha2\int\eta\zeta
       +\frac\beta4\int\zeta^2,
\end{aligned}
\end{equation}
where
\[
q_c(u)=\int|u_x|^2-\int Q_cu^2,\qquad
m_c(\omega)=\int|\omega_x|^2
                 +\int(\alpha-1-\alpha Q_c)|\omega|^2.
\]
We use the constraint slice
\begin{equation}\label{coer:constraints}
\int R_c'\epsilon_1=0,\qquad
\int R_c'\epsilon_2=0,\qquad
\int R_cQ_c\epsilon_2=0.
\end{equation}
In particular, the linear term in \eqref{coer:expansion} vanishes
on this slice. It is not claimed that $(U_c,0)$ is an unconstrained
critical point of $H$ when $c\ne0$.

\subsection{A weighted spectral estimate}
\begin{lemma}\label{coer:weighted}
Let
\[
\langle u,v\rangle_c=\int Q_cuv,\qquad
e_1=\frac{R_c}{(\int R_c^2Q_c)^{1/2}}.
\]
For real $u\in H_c^1$ with $\int Q_cu=0$,
\begin{equation}\label{coer:gap}
q_c(u)\ge
2\left(\int Q_cu^2-|\langle u,e_1\rangle_c|^2\right)\ge0.
\end{equation}
If also $\int Q_cR_cu=0$, then $q_c(u)\ge2\int Q_cu^2$.
\end{lemma}
\begin{proof}
Set $y=\tanh(\gamma x/\sqrt2)$ and $w(y)=u(x)$. Then
\[
\int Q_cu^2=\sqrt2\gamma\int_{-1}^1w^2,\qquad
q_c(u)=\sqrt2\gamma\left(
 \frac12\int_{-1}^1(1-y^2)|w'|^2-\int_{-1}^1w^2\right).
\]
Let $P_n$ be the $L^2(-1,1)$-normalized Legendre polynomials.
They form a complete orthonormal basis and solve
$-((1-y^2)P_n')'=n(n+1)P_n$.
The integration-by-parts identity
\[
 \int_{-1}^1(1-y^2)w'P_n'
   =n(n+1)\int_{-1}^1wP_n
\]
holds on the full form domain
$w\in L^2(-1,1)$, $\sqrt{1-y^2}\,w'\in L^2(-1,1)$.
Indeed, local Sobolev control and the fundamental theorem of calculus
give $|u(x)|\le C_u(1+|x|^{1/2})$.
In the variable $x$, the factor $1-y^2$ decays exponentially and
$P_n'(y)$ is bounded. Thus the boundary terms
$(1-y^2)wP_n'$ tend to zero at both endpoints. Both integrals are
finite by Cauchy--Schwarz, so integration on compact subintervals
followed by passage to the endpoints proves the identity.
For $n\ge1$, the functions
$\sqrt{1-y^2}P_n'/\sqrt{n(n+1)}$ are orthonormal.
Writing $a_n=\int_{-1}^1wP_n$, Bessel's inequality gives
\[
 \int_{-1}^1(1-y^2)|w'|^2
       \ge\sum_{n\ge1}n(n+1)|a_n|^2.
\]
Orthogonality to $1$ gives $a_0=0$; the coefficient of $|a_1|^2$
in the quadratic form is zero, and all coefficients for $n\ge2$
are at least $2$. The normalized image of $P_1$ is $e_1$.
This proves both assertions without requiring boundary values of $w$.
\end{proof}

\subsection{The dark-component estimate}
\begin{lemma}\label{coer:dark}
Fix a compact interval $K\subset(-\sqrt2,\sqrt2)$ and $0<a\le1/4$.
There exist $b,C>0$, uniform for $c\in K$, such that every
$(U_c+\epsilon,0)\in\mathcal X$ satisfying \eqref{coer:constraints}
obeys
\begin{equation}\label{coer:dark-bound}
\frac12q_c(\epsilon_1)+\frac12q_c(\epsilon_2)
      +a\|\eta\|_2^2
\ge b\big(\|\epsilon\|_{H_c^1}^2+\|\eta\|_2^2\big)
              -C\|\epsilon\|_{H_c^1}^3.
\end{equation}
\end{lemma}
\begin{proof}
Denote the left-hand side by $S$, and set
\[
A=\int R_c^2Q_c\epsilon_1^2,\qquad
W_j=\int Q_c\epsilon_j^2,\qquad N=\|\epsilon\|_{H_c^1}.
\]
The weighted norm controls $H^1(-1,1)$ uniformly for $c\in K$.
The fundamental theorem of calculus therefore gives
$|\epsilon(x)|\le C_K(1+|x|^{1/2})N$, and the exponential decay
of $Q_c$ implies
\begin{equation}\label{coer:cubic}
\int Q_c|\epsilon|^3\le C_KN^3.
\end{equation}
Using $0<Q_c\le1$, estimate one half of $a\int\eta^2$ against $Q_c$.
To keep track of the quartic term, first write the exact identity
\[
 \eta^2=4(R_c\epsilon_1+I_c\epsilon_2)^2
       +4(R_c\epsilon_1+I_c\epsilon_2)|\epsilon|^2+|\epsilon|^4.
\]
Applying $(X+Y)^2\ge X^2/2-Y^2$ only to the first square,
with $X=R_c\epsilon_1$ and $Y=I_c\epsilon_2$, and using $|U_c|\le1$,
we obtain
\[
 \eta^2\ge
 2R_c^2\epsilon_1^2-4I_c^2\epsilon_2^2-4|\epsilon|^3+|\epsilon|^4.
\]
We now discard the nonnegative quartic term and use
\eqref{coer:cubic} to control the weighted cubic remainder.
Lemma~\ref{coer:weighted} gives
$q_c(\epsilon_1)\ge0$ and $q_c(\epsilon_2)/2\ge W_2$, so
\begin{equation}\label{coer:first-bound}
S\ge aA+\frac12W_2+\frac a2\|\eta\|_2^2-CN^3.
\end{equation}
Here $1-2aI_c^2\ge1/2$ was used.

The factor $R_c^2$ vanishes at the origin, so the first term in
\eqref{coer:first-bound} does not itself control $W_1$.
To obtain this missing control, note that
$S\ge q_c(\epsilon_1)/2$ and apply \eqref{coer:gap}:
\[
S\ge W_1-|\langle\epsilon_1,e_1\rangle_c|^2.
\]
The remaining zero-mode coefficient is controlled by $A$, since
\begin{align*}
|\langle\epsilon_1,e_1\rangle_c|^2
 &=\frac{(\int \epsilon_1R_cQ_c)^2}{\int R_c^2Q_c}\\
 &\le\frac{\int Q_c}{\int R_c^2Q_c}\,A
   =\Gamma_cA,\qquad \Gamma_c=\frac3{\gamma^2}.
\end{align*}
Thus a second lower bound for the same quantity is
$S\ge W_1-\Gamma_cA$.
Taking a convex combination of the first bound, with weight
$1-\vartheta$, and the second, with weight $\vartheta$, rather than
counting $S$ twice, gives
\[
\begin{aligned}
S\ge{}&\big((1-\vartheta)a-\vartheta\Gamma_c\big)A
       +\vartheta W_1+\frac{1-\vartheta}{2}W_2\\
 &+\frac{(1-\vartheta)a}{2}\|\eta\|_2^2-CN^3.
\end{aligned}
\]
Choose $0<\vartheta\le a/(2(a+\Gamma_c))$. The coefficient of $A$
is nonnegative and can be discarded. Consequently,
\begin{equation}\label{est:10A}
S\ge c_4\left(\int Q_c|\epsilon|^2+\|\eta\|_2^2\right)-CN^3,
\qquad
c_4=\min\left\{\vartheta,\frac{1-\vartheta}{2},
                      \frac{(1-\vartheta)a}{2}\right\}>0.
\end{equation}
This clarifies the roles of the two complementary controls: the
spectral gap controls the component perpendicular to $e_1$, while $A$
controls the missing coefficient.

Finally, $S\ge\frac12\|\epsilon_x\|_2^2
-\frac12\int Q_c|\epsilon|^2$.
A convex combination of this bound and \eqref{est:10A}, with
sufficiently small weight on the former, proves
\eqref{coer:dark-bound}. All constants are uniform on $K$, because
$\gamma$ is bounded away from zero.
\end{proof}

\subsection{Mass-augmented coercivity}
Set
\[
J=H+2M,\qquad
h(c)=H(U_c,0)=\frac13(2-c^2)^{3/2}.
\]
\begin{proposition}\label{prop:augmented-coercivity}
Fix a compact interval $K\subset(-\sqrt2,\sqrt2)$.
There exist $\rho,\kappa,B>0$ such that, for $c\in K$,
$(U_c+\epsilon,\omega)\in\mathcal X$,
$\|\epsilon\|_{H_c^1}+\|\omega\|_{H^1}<\rho$, and
\eqref{coer:constraints}, one has
\begin{equation}\label{coer:augmented}
\kappa D_c^2
\le J(U_c+\epsilon,\omega)-h(c)\le B D_c^2,
\qquad
D_c^2=\|\epsilon\|_{H_c^1}^2+\|\omega\|_{H^1}^2+\|\eta\|_2^2.
\end{equation}
Constants may depend on $K,\alpha,\beta$.
\end{proposition}
\begin{proof}
Keep the full bright contribution together. The exact identity
\[
\begin{aligned}
 J(U_c+\epsilon,\omega)-h(c)
 ={}&H(U_c+\epsilon,0)-h(c)+\frac12\|\omega_x\|_2^2\\
 &+\frac12\int(1+\alpha|U_c+\epsilon|^2)|\omega|^2
       +\frac\beta4\|\omega\|_4^4
\end{aligned}
\]
and $\alpha,\beta\ge0$ imply
\[
 J(U_c+\epsilon,\omega)-h(c)
 \ge H(U_c+\epsilon,0)-h(c)+\frac12\|\omega\|_{H^1}^2.
\]
On the slice \eqref{coer:constraints} the scalar linear term vanishes.
Lemma~\ref{coer:dark}, applied with $a=1/4$, therefore gives
\[
 J(U_c+\epsilon,\omega)-h(c)
 \ge b\bigl(\|\epsilon\|_{H_c^1}^2+\|\eta\|_2^2\bigr)
       +\frac12\|\omega\|_{H^1}^2-C\|\epsilon\|_{H_c^1}^3.
\]
Decrease $\rho$ to absorb the cubic remainder. This proves the lower
bound in \eqref{coer:augmented}, without requiring $\beta>\alpha^2$
or $\alpha>1$, and includes $\alpha=0$ or $\beta=0$.
For the upper bound use the exact expansion,
$|\int\eta|\omega|^2|\le C\|\eta\|_2\|\omega\|_{H^1}^2$, and
$\|\omega\|_4^4\le C\|\omega\|_{H^1}^4$.
Young's inequality bounds the mixed term by
$C\|\eta\|_2^2+C\|\omega\|_{H^1}^4$.
\end{proof}

\begin{lemma}[Modulus and density defects]\label{coer:metric}
Uniformly for $c$ in a compact subsonic interval,
\[
\big\||U_c+\epsilon|-|U_c|\big\|_2^2
\le C\left(\|\epsilon\|_{H_c^1}^2+\|\eta\|_2^2\right).
\]
Conversely, in a sufficiently small neighbourhood for
$\mathfrak d_c$, one has
$\|\eta\|_2\le C\big\||U_c+\epsilon|-|U_c|\big\|_2$.
In particular, $D_c$ and $\mathfrak d_c(U_c+\epsilon,\omega)$
are locally comparable, including at $c=0$.
\end{lemma}
\begin{proof}
Choose uniform $q,m>0$ with $Q_c\ge q$ on $[-1,1]$ and
$|U_c|\ge m$ on its complement. On the first interval use the
reverse triangle inequality; outside it use
\[
\big||U_c+\epsilon|-|U_c|\big|
=\frac{|\eta|}{|U_c+\epsilon|+|U_c|}
\le m^{-1}|\eta|.
\]
For the converse, put $r=|U_c+\epsilon|$ and $r_c=|U_c|$.
The derivative bound
$\|(r-r_c)'\|_2\le\|\epsilon_x\|_2+2\|U_c'\|_2$
and the assumed smallness of $\|r-r_c\|_2$ give a uniform bound on
$r$ by one-dimensional Sobolev interpolation.
Now $\eta=(r-r_c)(r+r_c)$ gives the claim.
\end{proof}

\subsection{The unaugmented energy threshold}
The restriction in the following proposition concerns $H$ alone,
not the stability theorem.
Define
\[
 f(\gamma)=\frac{2(s+5)}{(s+1)^2},
 \qquad s=\sqrt{9-8\gamma^2},\qquad 0\le\gamma\le1.
\]
This is the continuous extension of
$\frac18[((-\gamma^2+\sqrt{9-8\gamma^2})/(1-\gamma^2))^2-1]$
from $0\le\gamma<1$. In particular, $f(0)=1$, $f(1)=3$, and
\[
 f'(\gamma)=\frac{16\gamma(s+9)}{s(s+1)^3}>0
 \qquad(0<\gamma<1).
\]
\begin{proposition}\label{coer:threshold}
Fix $|c|<\sqrt2$ and coefficients satisfying \eqref{eq:parameters}.
If $\alpha>f(\gamma_c)$, then for all $(U_c+\epsilon,\omega)\in\mathcal X$
with sufficiently small
$\|\epsilon\|_{H_c^1}+\|\omega\|_{H^1}$ satisfying
\eqref{coer:constraints},
\[
H(U_c+\epsilon,\omega)-h(c)\ge C
\left(\|\epsilon\|_{H_c^1}^2+\|\omega\|_{H^1}^2
      +\big\||U_c+\epsilon|-|U_c|\big\|_2^2\right).
\]
The constants depend on the fixed parameters.
If $1<\alpha<f(\gamma_c)$, the same slice contains arbitrarily
small nonzero perturbations along which $H$ decreases, and others
along which it increases. If $\alpha=f(\gamma_c)$, no positive
quadratic coercivity constant of the displayed form exists.
\end{proposition}
\begin{proof}
It suffices to consider $\alpha>1$: all the asserted alternatives
have this property, since $f(\gamma_c)>1$ for $|c|<\sqrt2$.
Put
\[
p_\alpha=\frac{\sqrt{1+8\alpha}-1}{2},\qquad
\varphi(x)=\sech^{p_\alpha}\!\left(\frac{\gamma x}{\sqrt2}\right),\qquad
\lambda_0=\alpha-1-\frac{\gamma^2p_\alpha^2}{2}.
\]
A direct differentiation shows that
$(-\partial_x^2+\alpha-1-\alpha Q_c)\varphi=\lambda_0\varphi$.
The positive ground state gives the factorization
\[
m_c(w)=\int\varphi^2\left|\partial_x(w/\varphi)\right|^2
                  +\lambda_0\|w\|_2^2
\]
initially for compactly supported smooth $w$.
Density gives $m_c(w)\ge\lambda_0\|w\|_2^2$ for every $w\in H^1$.
Since $\varphi\in H^1$, $\lambda_0$ is the bottom of the spectrum.
Solving $\lambda_0>0$ gives exactly $\alpha>f(\gamma)$;
at $\gamma=1$ it reduces to $\alpha>3$.
If $\lambda_0>0$, boundedness of the potential also gives
$m_c(w)\ge b\|w\|_{H^1}^2$ for some $b>0$.
In this case, the mixed term can be estimated by
\[
 \frac14\eta^2+\frac\alpha2\eta|\omega|^2
          +\frac\beta4|\omega|^4
 \ge\frac18\eta^2-\frac{\alpha^2}{2}|\omega|^4,
\]
using $\beta\ge0$. Apply Lemma~\ref{coer:dark} with $a=1/8$ in
\eqref{coer:expansion}, where the linear term vanishes on the slice.
Together with $m_c(\omega)\ge b\|\omega\|_{H^1}^2$, this gives
\[
\begin{aligned}
 H(U_c+\epsilon,\omega)-h(c)
 \ge{}&b_1\bigl(\|\epsilon\|_{H_c^1}^2+\|\eta\|_2^2\bigr)
        +b_2\|\omega\|_{H^1}^2\\
 &-C\|\epsilon\|_{H_c^1}^3-C_\alpha\|\omega\|_{H^1}^4
\end{aligned}
\]
for some $b_1,b_2>0$. Smallness absorbs both remainders, and
Lemma~\ref{coer:metric} yields the asserted modulus-defect bound.

If $\lambda_0<0$, take $\epsilon=0$ and $\omega=t\varphi$.
The exact energy difference is
\[
\frac{\lambda_0t^2}{2}\|\varphi\|_2^2
     +\frac{\beta t^4}{4}\|\varphi\|_4^4<0
\]
for sufficiently small $t>0$. For a positive direction take
$\omega=0$ and $\epsilon=t g$, where $g$ is a nonzero real odd
smooth compactly supported function; Lemma~\ref{coer:dark} applies.
These are directions on the constraint slice, not an instability
criterion. At $\lambda_0=0$, the same bright perturbation gives
\[
 H(U_c,t\varphi)-h(c)=\frac{\beta t^4}{4}\|\varphi\|_4^4.
\]
Dividing by $t^2$ and letting $t\to0$ rules out any positive quadratic
coercivity constant, also when $\beta=0$. The positive coercivity
constant cannot be uniform as $\lambda_0\downarrow0$.
\end{proof}

In terms of the parameter $p_\alpha$ above,
\[
 \lambda_0=\frac{p_\alpha-2}{2}+\frac{p_\alpha^2c^2}{4}.
\]
Consequently, if $\alpha>3$, unaugmented coercivity holds at every
subsonic speed. If $\alpha=3$, it holds precisely for
$0<|c|<\sqrt2$; the black soliton lies at the threshold.
For $1<\alpha<3$, define
\[
 c_*(\alpha)=\frac{\sqrt{2(2-p_\alpha)}}{p_\alpha}.
\]
Coercivity holds for $c_*(\alpha)<|c|<\sqrt2$, the lowest bright
eigenvalue is zero for $|c|=c_*(\alpha)$, and it is negative for
$|c|<c_*(\alpha)$. These alternatives concern the unaugmented energy,
not the orbital stability proved in Theorem~\ref{thm:stability}.

\section{The total untwisted momentum}\label{sec:momentum}

The dark component need not have a global phase, or phase limits at
either end of the line. We therefore use a momentum with values in
$\mathbb R/\pi\mathbb Z$. The scalar construction and its affine
variation formula are due to B\'ethuel--Gravejat--Saut--Smets
\cite[Lemma~1.9 and Proposition~1.16]{BGSS2008}. We give the details
needed for the coupled flow, including the topology used in passing
to nonsmooth solutions.

For $u\in\mathcal E$, the function $|u|-1$ belongs to $H^1(\mathbb R)$:
indeed $||u|-1|\le |1-|u|^2|$ and
$|\partial_x|u||\le |u_x|$ almost everywhere. Consequently
$|u(x)|\to1$ as $|x|\to\infty$. Choose $R>0$ such that $|u|\ge1/2$
on $\{|x|\ge R\}$, and choose continuous phases $\vartheta_+$ and
$\vartheta_-$ on the two exterior half-lines. The common notation for their derivatives is
\[
 v_u=\frac{\operatorname{Im}(\bar u u_x)}{|u|^2}
 \quad (|x|\ge R).
\]
In particular $v_u\in L^2(\{|x|\ge R\})$. Define
\begin{equation}\label{mom:finite-core}
\begin{split}
 \mathcal P_R(u)={}&\frac12\int_{-R}^{R}
                       \operatorname{Im}(\bar u u_x)
 -\frac12\bigl(\vartheta_+(R)-\vartheta_-(-R)\bigr)\\
 &+\frac12\int_{|x|>R}(|u|^2-1)v_u,
 \qquad [P_{\rm d}](u)=\mathcal P_R(u)\pmod\pi.
\end{split}
\end{equation}
Every integral is finite, the exterior one by Cauchy--Schwarz.
Increasing $R$ does not change $\mathcal P_R$ if the same phase
branches are used: the change in the interior integral cancels the
change in the phase increment and exterior integral. Changing either
phase branch changes $\mathcal P_R$ by an integer multiple of $\pi$.
Thus $[P_{\rm d}]$ is well defined. Equivalently, it is the limit
modulo $\pi$ of
\[
 \frac12\int_{-S}^{S}\operatorname{Im}(\bar u u_x)
 -\frac12\bigl(\vartheta_+(S)-\vartheta_-(-S)\bigr),
 \qquad S\longrightarrow\infty.
\]
This expression also proves invariance under translations and
multiplication by a constant complex number of modulus one.
For $(u,w)\in\mathcal X$ set
\begin{equation*}
 [P](u,w)=[P_{\rm d}](u)
       +\frac12\int\operatorname{Im}(\bar w w_x)
       \quad\hbox{in }\mathbb R/\pi\mathbb Z.
\end{equation*}

\begin{proposition}[Conservation of total momentum]\label{prop:momentum}
The map $[P]:\mathcal X\to\mathbb R/\pi\mathbb Z$ is continuous for
$d_{\mathcal X}$. Every solution furnished by Theorem~\ref{thm:gwp}
satisfies
\[
 [P](\psi(t),\phi(t))=[P](\psi_0,\phi_0),\qquad t\in\mathbb R.
\]
Moreover, for $u\in\mathcal E$ and $h\in H^1(\mathbb R;\mathbb C)$,
\begin{equation}\label{mom:affine}
 [P_{\rm d}](u+h)-[P_{\rm d}](u)
  =\frac12\int\operatorname{Im}(\bar h h_x)
     +\int\operatorname{Im}(\bar h u_x)\pmod\pi.
\end{equation}
On the open subset where $\inf_{\mathbb R}|u|>0$, the real quantity
\begin{equation}\label{mom:nonvanishing}
 P(u,w)=\frac12\int (|u|^2-1)
                \frac{\operatorname{Im}(\bar u u_x)}{|u|^2}
          +\frac12\int\operatorname{Im}(\bar w w_x)
\end{equation}
is a continuous lift of $[P]$. It is conserved on every connected
time interval on which the solution has a nonvanishing dark component.
\end{proposition}

\begin{proof}
\emph{Continuity.} The continuity of the scalar untwisted momentum $[P_{\rm d}]$ is classical; see \cite[Lemma~2.7]{BGSS2008}. We recall the argument in our stronger energy-space topology for completeness.
Let $u_n\to u$ in the dark-component part of $d_{\mathcal X}$.
Uniform convergence provides one common exterior region on which
$|u|,|u_n|\ge1/3$. On this region,
\[
 \|v_{u_n}-v_u\|_2
 \le C\bigl(\|u_{nx}-u_x\|_2
          +\|u_n-u\|_\infty\|u_x\|_2\bigr)\longrightarrow0.
\]
The density defects converge in $L^2$, so the exterior integrals in
\eqref{mom:finite-core} converge. The interior integrals converge by
uniform convergence of the functions and $L^2$ convergence of their
derivatives. At the two endpoints choose phase representatives
converging to those of $u$. This proves continuity of $[P_{\rm d}]$;
the bright term is continuous in $H^1$.

\emph{Affine variation.}
The function $u+h$ belongs to $\mathcal E$, and $h(x)\to0$ at both
ends. On a sufficiently remote exterior region the quotient
$(u+h)/u$ is close to one. Its argument can therefore be chosen
continuously and tends to zero at both ends. Subtract the truncated
momentum expressions above for $u+h$ and $u$. Integration by parts
on $[-S,S]$ gives
\[
\begin{split}
 &\frac12\int_{-S}^{S}
  \operatorname{Im}\bigl(\bar u h_x+\bar h u_x+\bar h h_x\bigr)\\
 &\qquad=\int_{-S}^{S}\operatorname{Im}(\bar h u_x)
       +\frac12\int_{-S}^{S}\operatorname{Im}(\bar h h_x)
       +\frac12[\operatorname{Im}(\bar u h)]_{-S}^{S}.
\end{split}
\]
The boundary term and the difference of exterior phase increments
tend to zero. The two remaining integrands are in $L^1$, proving
\eqref{mom:affine}.

\emph{Regular solutions.}
First consider initial data with $\psi_{0x}\in H^2$ and $\phi_0\in H^3$,
as in the regularization argument of Theorem~\ref{thm:gwp}.
Writing $h(t)=\psi(t)-\psi_0$, that theorem and the extended linear
group give
\[
 h\in C(\mathbb R;H^1),\qquad
 h_t=\psi_t\in C(\mathbb R;H^1).
\]
For example, the linear increment belongs to $H^1$ since
$(e^{it\xi^2}-1)/\xi$ is bounded, and the nonlinear Duhamel term
belongs to $H^1$. The higher regularity follows from the persistence
estimates already proved there. Formula~\eqref{mom:affine}, together
with the bright contribution, defines a differentiable real lift of
the momentum difference from time zero. Its derivative is
\begin{equation}\label{mom:derivative}
 \frac{dP}{dt}
 =\int\operatorname{Im}(\bar\psi_t\psi_x)
      +\int\operatorname{Im}(\bar\phi_t\phi_x).
\end{equation}
The integration by parts in differentiating
$\frac12\int\operatorname{Im}(\bar h h_x)$ is legitimate in $H^1$.

Put $r=|\psi|^2$, $s=|\phi|^2$, and
\[
 V(r,s)=\frac14(1-r)^2+\frac12(\alpha r-1)s+\frac\beta4s^2.
\]
Using the signs in \eqref{eq:system}, the integrand on the right of
\eqref{mom:derivative} is
\[
 \partial_x\left\{
  \frac12(|\psi_x|^2+|\phi_x|^2)-V(r,s)\right\}.
\]
The expression in braces is in $L^1$. Its derivative is integrable
for the present regularity; alternatively, compact cutoffs give an
error bounded by $C S^{-1}$ times its $L^1$ norm. Its integral is
therefore zero, and the total momentum is conserved for regular data.
Notice that the nonlinear coupling contributions cancel only after
the two components have been added.

\emph{General solutions.}
Mollification and the continuous-dependence assertion of
Theorem~\ref{thm:gwp} give regular solutions converging to any
energy-space solution in $C([-T,T];\mathcal X)$ for every $T<\infty$.
The continuity proved above transfers their momentum conservation to
the limit. Since $T$ is arbitrary, the conclusion holds for all time.

\emph{The nonvanishing chart.}
When $\inf|u|>0$, choose a continuous phase on the whole line. Its
derivative is $v_u\in L^2$. Substituting
$\vartheta(R)-\vartheta(-R)=\int_{-R}^R v_u$ into
\eqref{mom:finite-core} proves that \eqref{mom:nonvanishing}
is a real representative of $[P]$. The same $L^2$ product estimates
as in the continuity argument above apply on the whole line, so this
representative is continuous in the nonvanishing chart. Along a
connected interval it is a continuous real function with constant
class modulo $\pi$, and is therefore constant.
\end{proof}

\section{Orbital stability at nonzero speed}\label{sec:nonzero}

Fix $0<|c|<\sqrt2$ throughout this section. The reference speed will
not be modulated. The scalar argument is the energy--momentum
coercivity mechanism of \cite[Proposition~4 and Section~2.1]{BGS2014};
compare also the orbital-stability statement recalled in
\cite[Theorem~1]{BGS2015}. We give the required spectral argument and
coordinate estimates below. In the coupled problem, adding a positive
multiple of the separately conserved bright mass makes the transverse
quadratic form positive. Thus positivity of the unaugmented bright
energy is not required for orbital stability.

\subsection{Coordinates and the distance to the orbit}

For a nonvanishing dark component write
\[
 \psi=re^{i\vartheta},\qquad r=|\psi|,\qquad
 v=\vartheta_x=\operatorname{Im}\frac{\psi_x}{\psi}.
\]
A continuous phase lift exists on the line; no phase limit at infinity
is imposed. Set
\[
 Y=H^1(\mathbb R;\mathbb R)\times L^2(\mathbb R;\mathbb R)
       \times H^1(\mathbb R;\mathbb C),\qquad
 \|(h,k,z)\|_Y^2=\|h\|_{H^1}^2+\|k\|_2^2+\|z\|_{H^1}^2.
\]
The coordinate domain consists of $(r,v,z)\in(1,0,0)+Y$ with
$\inf r>0$. Define
\begin{equation*}
\begin{aligned}
 r_c&=|U_c|=\sqrt{1-\gamma^2\sech^2(\gamma x/\sqrt2)},
 &v_c&=\frac{c(1-r_c^2)}{2r_c^2},\\
 \gamma&=\sqrt{1-c^2/2},
 &\underline r_c&=\inf r_c=\frac{|c|}{\sqrt2}.
\end{aligned}
\end{equation*}
Write $\mathbf q_c=(r_c,v_c,0)$ and
\[
 \mathcal D_c(r,v,z)=\inf_{a\in\mathbb R}
       \|(r(\cdot+a)-r_c,v(\cdot+a)-v_c,z(\cdot+a))\|_Y.
\]

\begin{lemma}[Comparison with the centered complex distance]
\label{nz:bridge}
There exist $\varepsilon_c,C_c>0$ with the following properties.
If $\mathfrak d_c(\psi,z)<\varepsilon_c$, then
\[
 \inf|\psi|\ge \underline r_c/2,\qquad
 \|(|\psi|-r_c,v-v_c,z)\|_Y
       \le C_c\mathfrak d_c(\psi,z).
\]
Conversely, if $\|(r-r_c,v-v_c,z)\|_Y<\varepsilon_c$, then every
reconstruction $\psi=re^{i\vartheta}$ with $\vartheta_x=v$ satisfies
\[
 \inf_{\theta\in\mathbb R}\mathfrak d_c(e^{-i\theta}\psi,z)
       \le C_c\|(r-r_c,v-v_c,z)\|_Y.
\]
Consequently the two orbit distances are locally comparable near the
fixed profile, with translation applied to the solution before
evaluating $\mathfrak d_c$.
\end{lemma}

\begin{proof}
Put $d=\mathfrak d_c(\psi,z)$, $e=\psi-U_c$, and $h=|\psi|-r_c$.
Since $Q_c$ has a positive lower bound on $[-1,1]$, the local Sobolev
inequality and the fundamental theorem of calculus give
\begin{equation}\label{nz:pointwise-error}
 |e(x)|\le C_c(1+\sqrt{|x|})\|e\|_{H_c^1}
          \le C_c(1+\sqrt{|x|})d.
\end{equation}
Initially we need only the rough estimate
\[
 \|h\|_2\le d,\qquad
 \|h_x\|_2\le\|\psi_x\|_2+\|U_c'\|_2
                   \le d+2\|U_c'\|_2.
\]
Thus $\|h\|_\infty^2\le2d(d+2\|U_c'\|_2)$. For small $d$ this
implies $r=|\psi|\ge \underline r_c/2$, as well as a uniform upper bound for $r$.
This preliminary square-root estimate is used only to establish
nonvanishing; the coordinate estimate itself is linear in $d$.

Let $A=\bar\psi/r$ and $A_c=\bar U_c/r_c$. In this amplitude range,
$|A-A_c|\le C_c|e|$, and therefore
\[
 h_x=\operatorname{Re}\big((\psi_x-U_c')A+U_c'(A-A_c)\big).
\]
By \eqref{nz:pointwise-error} and exponential decay of $U_c'$,
$\|h_x\|_2\le C_cd$. Similarly,
\[
 v-v_c=\operatorname{Im}\left(
       \frac{\psi_x-U_c'}{\psi}
       +U_c'\left(\frac1\psi-\frac1{U_c}\right)\right),
 \qquad
 \left|\frac1\psi-\frac1{U_c}\right|\le C_c|e|.
\]
Hence $\|v-v_c\|_2\le C_cd$, which proves the first assertion.

Conversely, set $h=r-r_c$, $k=v-v_c$, and choose the constant phase so
that $\vartheta(0)=\vartheta_c(0)$, where $U_c=r_ce^{i\vartheta_c}$.
For $\Delta\vartheta=\vartheta-\vartheta_c$ we have
\[
 |\Delta\vartheta(x)|\le\sqrt{|x|}\|k\|_2.
\]
The identities
\begin{align*}
 \psi-U_c&=e^{i\vartheta}h+
               (e^{i\vartheta}-e^{i\vartheta_c})r_c,\\
 \psi_x-U_c'&=e^{i\vartheta}
                  \big(h_x+i(rk+hv_c)\big)
       +(e^{i\vartheta}-e^{i\vartheta_c})(r_c'+ir_cv_c)
\end{align*}
give
\[
 \|\psi-U_c\|_{H_c^1}
       \le C_c(\|h\|_{H^1}+\|k\|_2),
\]
using $H^1\hookrightarrow L^\infty$ and
$\int |x|(Q_c+|U_c'|^2)<\infty$. The modulus difference is exactly
$h$, and the bright norm is unchanged. Taking approximate minimizing
translations and phases proves the orbit-distance comparison. No
translation invariance of the weighted norm is asserted or needed.
\end{proof}

\subsection{The scalar constrained Hessian}

In these coordinates the conserved functionals are
\begin{align}
 H(r,v,z)&=\frac12\int(r_x^2+r^2v^2+|z_x|^2)
       +\frac14\int(1-r^2)^2
       +\frac12\int(\alpha r^2-1)|z|^2
       +\frac\beta4\int|z|^4,\label{nz:energy}\\
 P(r,v,z)&=\frac12\int(r^2-1)v
       +\frac12\int\operatorname{Im}(\bar z z_x),\qquad
 M(z)=\frac12\int|z|^2.\label{nz:momentum}
\end{align}
Here $P$ is the real momentum in the nonvanishing chart, conserved by
Proposition~\ref{prop:momentum}. All integrals are absolutely convergent,
and these are smooth real functionals on the coordinate domain.
Set $E(r,v)=H(r,v,0)$ and $P_{\rm d}(r,v)=P(r,v,0)$.
The stationary scalar action, with the convention of
\eqref{eq:system}, is $S_c=E+cP_{\rm d}$. Completing the square gives
\begin{align}
 S_c(r,v)&=\frac12\int r_x^2+\int W_c(r)
        +\frac12\int r^2(v-f_c(r))^2,\nonumber\\
 f_c(r)&=\frac c2(r^{-2}-1),\qquad
 W_c(r)=\frac14(r^2-1)^2
             -\frac{c^2}{8}(r^2-2+r^{-2}).\nonumber
\end{align}
The profile satisfies $v_c=f_c(r_c)$ and
$-r_c''+W_c'(r_c)=0$, so $DS_c(r_c,v_c)=0$.

For $w=(h,k)\in Y_{\rm d}=H^1(\mathbb R;\mathbb R)
\times L^2(\mathbb R;\mathbb R)$ the second variation is
\begin{align}
 q_{{\rm d},c}(w)&=\int(h_x^2+V_ch^2)
           +\int r_c^2(k+cr_c^{-3}h)^2,\label{nz:hessian}\\
 V_c&=3r_c^2-1-\frac{c^2}{4}-\frac{3c^2}{4r_c^4}.\nonumber
\end{align}
Let $B_c$ be the associated bilinear form and define
\[
 \ell_c(h,k)=DP_{\rm d}(r_c,v_c)[h,k]
       =\int r_cv_ch+\frac12\int(r_c^2-1)k,
 \qquad F_c(h,k)=\int h r_c'.
\]

\begin{lemma}[Constrained coercivity]\label{nz:scalar-coercivity}
There exists $\kappa_c>0$ such that
\[
 q_{{\rm d},c}(w)\ge\kappa_c\|w\|_{Y_{\rm d}}^2
      \quad\text{if }\ell_c(w)=F_c(w)=0.
\]
Moreover, if $K_c>1/(2\sqrt{2-c^2})$, then for some
$\widetilde\kappa_c>0$,
\[
 q_{{\rm d},c}(w)+2K_c\ell_c(w)^2
      \ge\widetilde\kappa_c\|w\|_{Y_{\rm d}}^2
      \quad\text{if }F_c(w)=0.
\]
\end{lemma}

\begin{proof}
Consider $L_c=-\partial_x^2+V_c$ on real $L^2$, with domain $H^2$.
Its potential is smooth and bounded, and tends exponentially to
$2-c^2>0$ at both ends. Thus its essential spectrum is
$[2-c^2,\infty)$. Differentiating the profile equation gives
$L_cr_c'=0$. The exponentially decaying eigenfunction $r_c'$ has
exactly one zero, at the origin. The one-dimensional Sturm oscillation
theorem implies that zero is a simple eigenvalue and exactly one
simple negative eigenvalue lies below it.

The bounded invertible transformation
$(h,k)\mapsto(h,k+cr_c^{-3}h)$ in $Y_{\rm d}$ reduces
\eqref{nz:hessian} to the scalar form of $L_c$ and a positive square.
Consequently $q_{{\rm d},c}$ has negative index one, and the nullspace of its
bilinear form is spanned by $Z_c=(r_c',v_c')$.
Put $W_c^{\rm par}=\partial_c(r_c,v_c)$. Differentiating stationarity
and integrating the explicit momentum yield
\begin{gather}
 B_c(W_c^{\rm par},w)=-\ell_c(w),\label{nz:parameter-identity}\\
 p_c=P_{\rm d}(r_c,v_c)
       =\frac{c\gamma}{\sqrt2}
        -\operatorname{sgn}(c)\arccos\frac{|c|}{\sqrt2},
 \qquad p_c'=\sqrt{2-c^2}>0.\nonumber
\end{gather}
The parameter derivative is taken on either open nonzero-speed
interval. In particular,
$\ell_c(W_c^{\rm par})=p_c'$ and
$q_{{\rm d},c}(W_c^{\rm par})=-p_c'<0$.
If $\ell_c(w)=0$, then $w$ is $B_c$-orthogonal to this negative
direction. A negative value of $q_{{\rm d},c}(w)$ would give two independent
negative directions. Therefore $q_{{\rm d},c}\ge0$ on $\ker\ell_c$.

Suppose $w\in\ker\ell_c$ and $q_{{\rm d},c}(w)=0$. Nonnegativity applied to
$w+ts$, $s\in\ker\ell_c$, gives $B_c(w,s)=0$.
Hence $B_c(w,\cdot)=\lambda\ell_c$. Testing against
$W_c^{\rm par}$ and using \eqref{nz:parameter-identity} gives
$\lambda p_c'=0$. Thus $w$ belongs to the nullspace spanned by $Z_c$.
If also $F_c(w)=0$, then $w=0$, since
$F_c(Z_c)=\|r_c'\|_2^2>0$.

To obtain a uniform gap on $Y_0=\ker\ell_c\cap\ker F_c$, observe that
\eqref{nz:hessian} and $V_c\to2-c^2>0$ imply, for some $A,R,\kappa>0$,
\begin{equation}\label{nz:garding}
 q_{{\rm d},c}(w)+A\|h\|_{L^2(-R,R)}^2\ge\kappa\|w\|_{Y_{\rm d}}^2.
\end{equation}
If the asserted gap failed, a normalized sequence $w_n\in Y_0$
with $q_{{\rm d},c}(w_n)\to0$
has a weak subsequential limit $w$, with local strong $L^2$ convergence
of the first component. The positive form on the left of
\eqref{nz:garding} is weakly lower semicontinuous; subtracting the
convergent local term gives $q_{{\rm d},c}(w)\le0$. The limit belongs to $Y_0$,
so $q_{{\rm d},c}(w)=0$, while \eqref{nz:garding} forces $w\ne0$. This
contradicts the preceding paragraph and proves the first assertion.

Finally, $F_c(W_c^{\rm par})=0$ by parity. For $F_c(w)=0$, write
\[
 w=w_0+\frac{\ell_c(w)}{p_c'}W_c^{\rm par},\qquad w_0\in Y_0.
\]
The summands are $B_c$-orthogonal, and hence
\[
 q_{{\rm d},c}(w)+2K_c\ell_c(w)^2
       =q_{{\rm d},c}(w_0)+\left(2K_c-\frac1{p_c'}\right)\ell_c(w)^2.
\]
The decomposition is bounded in $Y_{\rm d}$, proving the second claim.
\end{proof}

\subsection{The bright-mass correction and global stability}

Choose $K_c$ as in Lemma~\ref{nz:scalar-coercivity}, set
$\Lambda_c=2+c^2/4$, and introduce the fixed functional
\begin{equation}\label{nz:lyapunov}
 \mathcal L_c=H+cP+\Lambda_cM+K_c(P-p_c)^2.
\end{equation}
The scalar squared-momentum correction follows the Lyapunov construction
in Chiron \cite[Theorem~12 and Appendix~A]{Chiron2013}, expressed in
our sign convention. The additional bright-mass term is used here
to control perturbations in the second component.
The functional $\mathcal L_c$ is conserved on every time interval on
which the dark component remains uniformly nonvanishing on compact time
subintervals. Its first
derivative vanishes at $\mathbf q_c$. Since the bright component of
the reference wave is zero, the Hessian has no mixed dark--bright
terms. Its bright part is
\begin{align}
 q_{{\rm b},c}(z)
 &=\int\bigl(|z_x|^2+c\operatorname{Im}(\bar z z_x)
                    +(\alpha r_c^2-1+\Lambda_c)|z|^2\bigr)\nonumber\\
 &=\int\left|z_x+\frac{ic}{2}z\right|^2
             +\int(\alpha r_c^2+1)|z|^2
       \ge\kappa_{{\rm b},c}\|z\|_{H^1}^2.\label{nz:bright}
\end{align}
The last inequality follows from $\alpha\ge0$ and
$\|z_x\|_2\le\|z_x+icz/2\|_2+|c|\|z\|_2/2$.
The independent mass conservation is essential here; the dark and
bright momenta need not be separately conserved.

Combining \eqref{nz:bright} with Lemma~\ref{nz:scalar-coercivity} and
Taylor's formula gives constants $b_c,B_c^{\rm up},\rho_c>0$ such that
\begin{equation}\label{nz:nonlinear-coercivity}
 b_c\|w\|_Y^2\le
       \mathcal L_c(\mathbf q_c+w)-\mathcal L_c(\mathbf q_c)
       \le B_c^{\rm up}\|w\|_Y^2
\end{equation}
whenever $w=(h,k,z)$, $\|w\|_Y<\rho_c$, and $\int hr_c'=0$.
Without the orthogonality condition the upper bound holds in absolute
value. Indeed, the quadratic term is
\[
 \frac12\bigl(q_{{\rm d},c}(h,k)+2K_c\ell_c(h,k)^2+q_{{\rm b},c}(z)\bigr),
\]
and the remainder is $O(\|w\|_Y^3)$ on a small ball. The latter follows
from $H^1\hookrightarrow L^\infty$, $L^2\cdot L^2\subset L^1$, and
the polynomial formulas \eqref{nz:energy}--\eqref{nz:momentum}.

\begin{proof}[Proof of Theorem~\ref{thm:stability} for $c\ne0$]
Let $\delta=d_c^{\rm orb}(\psi_0,\phi_0)$ be sufficiently small.
For $\delta>0$ choose an initial representative whose
$\mathfrak d_c$-distance is less than $2\delta$.
Lemma~\ref{nz:bridge} gives nonvanishing and
$\mathcal D_c(r_0,v_0,\phi_0)\le C_c\delta$. If $\delta=0$, the argument
below applies with any positive upper bound for $\delta$; the conclusion
then follows by letting that upper bound tend to zero.

We first describe the translation slice. Near $\mathbf q_c$ solve
\[
 \int\bigl(r(x+a)-r_c(x)\bigr)r_c'(x)\,dx=0.
\]
The derivative with respect to $a$ at the profile is
$\|r_c'\|_2^2>0$. The implicit function theorem gives a continuous
local choice with $|a|\le C_c\|r-r_c\|_{H^1}$. Transferring the
translation to the smooth test function justifies differentiation
without requiring a second derivative of $r$. By recentering at an
approximate nearest translate we obtain, uniformly in the small tube,
\begin{equation}\label{nz:slice}
 \|w\|_Y\le C_c\mathcal D_c(r,v,z),\qquad
 w=(r(\cdot+a)-r_c,v(\cdot+a)-v_c,z(\cdot+a)),\quad
 \int w_1r_c'=0.
\end{equation}
For the second component, translate its error isometrically and use
$\|v_c(\cdot+a)-v_c\|_2\le |a|\|v_c'\|_2$; no derivative of the
error $v-v_c$ is needed. The same reasoning applies to the first
component, while the bright $H^1$ norm is translation invariant.

Start the global energy-space solution supplied by
Theorem~\ref{thm:gwp}. As long as its hydrodynamical orbit distance
remains in a sufficiently small fixed tube, the Sobolev embedding
ensures $r\ge \underline r_c/2$. The coordinates and $\mathcal D_c$ are continuous there.
Continuity of the coordinates follows from their formulas:
the denominators have a common positive lower bound, derivatives
converge in $L^2$, and local uniform convergence together with a
tail truncation controls the bounded phase factors multiplying a
fixed $L^2$ derivative. The amplitude difference is controlled in
$L^2$ by the difference of squared amplitudes under the same lower
bound.
Moreover, the orbital distance satisfies
\[
 \bigl|\mathcal D_c(r,v,z)
       -\mathcal D_c(\widetilde r,\widetilde v,\widetilde z)\bigr|
 \le \bigl\|(r-\widetilde r,v-\widetilde v,z-\widetilde z)\bigr\|_Y.
\]
Indeed, for each fixed translation the triangle inequality and
translation invariance of the $Y$ norm give the corresponding
one-sided estimate. Taking infima over translations and interchanging
the two triples proves the claim, and hence continuity of
$\mathcal D_c$ without any compactness assumption on minimizing
translations.

On every such interval Proposition~\ref{prop:momentum} and
Theorem~\ref{thm:gwp} conserve \eqref{nz:lyapunov}.
Translation invariance, \eqref{nz:slice}, and
\eqref{nz:nonlinear-coercivity} give
\[
 b_c\|w(t)\|_Y^2
 \le\mathcal L_c(\psi_0,\phi_0)-\mathcal L_c(U_c,0)
 \le C_c\delta^2.
\]
Choose the initial threshold so that the improved estimate stays
strictly inside half the tube radius and gives
$\|r(t,\cdot+a)-r_c\|_\infty<\underline r_c/2$.
There is no first finite exit time: at a putative endpoint, global
existence, continuity, the uniform amplitude lower bound, and the
recentered translation chart extend the estimate. The same argument
works backward in time. Consequently
\[
 \sup_{t\in\mathbb R}\mathcal D_c(r(t),v(t),\phi(t))\le C_c\delta,
 \qquad \inf_{t,x}|\psi(t,x)|\ge \underline r_c/2.
\]
Applying the reverse estimate in Lemma~\ref{nz:bridge} to each
translated representative proves
$\sup_{t\in\mathbb R}d_c^{\rm orb}(\psi(t),\phi(t))\le C_c\delta$.
The speed $c$ remains fixed throughout, and no equality of initial
and soliton momenta has been assumed. Constants may deteriorate as
$c\to0$ or $|c|\to\sqrt2$; the zero-speed case requires the separate
argument below.
\end{proof}

\section{Stability at zero speed}\label{sec:black}

We prove the black-soliton case of Theorem~\ref{thm:stability}.
The scalar ingredients concerning modulation and the momentum
expansion are those of Gravejat--Smets
\cite[Propositions~2, 4 and~5]{GS2015}. They are applied only in
their stated neighbourhood of the black-soliton orbit. The coupled
extension uses the conserved bright mass, the augmented coercivity
of Proposition~\ref{prop:augmented-coercivity}, and the total momentum
of Proposition~\ref{prop:momentum}.

\subsection{Local coordinates and the energy-space distance}

For a perturbation $u=U_c+\epsilon$ write
\[
 \eta=|u|^2-|U_c|^2,
 \qquad
 D_c(\epsilon,w)^2=
  \|\epsilon\|_{H_c^1}^2+\|w\|_{H^1}^2+\|\eta\|_2^2.
\]
The weighted norms $H_c^1$ are uniformly equivalent when $c$ ranges
over a fixed small interval about zero. To see this, use the local
$H^1$ estimate and the fundamental theorem of calculus to obtain
\begin{equation}\label{black:pointwise}
 |f(x)|\le C\|f\|_{H_0^1}(1+|x|^{1/2}).
\end{equation}
Since the weights $Q_c$ have uniformly bounded moments of every
fixed order, this bounds $\int Q_c|f|^2$ by $C\|f\|_{H_0^1}^2$.
The converse follows in the same way, using a uniform positive lower
bound for $Q_c$ on $[-1,1]$.

We shall also use the local equivalence
\begin{equation}\label{black:distance-equivalence}
 D_c(\epsilon,w)\asymp
 \|\epsilon\|_{H_c^1}+\||U_c+\epsilon|-|U_c|\|_2+\|w\|_{H^1}.
\end{equation}
Here and below equivalence constants are uniform for small $|c|$,
in a sufficiently small neighbourhood of the profile.
Indeed, on $[-1,1]$ the modulus difference is bounded by $|\epsilon|$,
whereas outside this interval $|U_c|$ is bounded below. Thus
\[
 \||U_c+\epsilon|-|U_c|\|_2^2
 \le C\left(\int Q_c|\epsilon|^2+\|\eta\|_2^2\right).
\]
Conversely, put $q=|U_c+\epsilon|-|U_c|$. If
$\|\epsilon\|_{H_c^1}+\|q\|_2$ is small, then
\[
 \|q_x\|_2\le\|\epsilon_x\|_2+2\|U_c'\|_2\le C.
\]
The inequality $\|q\|_\infty^2\le2\|q\|_2\|q_x\|_2$ bounds
$|U_c+\epsilon|$ uniformly. Factoring $\eta$ now gives
$\|\eta\|_2\le C\|q\|_2$, as required.

For clarity, introduce the scalar metric used in \cite{GS2015},
\[
 d_0(u,U_0)^2=
  \|u-U_0\|_{H_0^1}^2+\||u|^2-|U_0|^2\|_2^2,
\]
and its orbital neighbourhood
\[
 \mathcal V_0(\rho)=
 \{u\in\mathcal E:\inf_{a,\theta}
                   d_0(e^{-i\theta}u(\cdot+a),U_0)<\rho\}.
\]
In particular, smallness in the distance of Theorem~\ref{thm:stability}
implies membership in this neighbourhood; it is not necessary to
assume smallness in $L^\infty$ of the complex perturbation on the
whole line.

\begin{lemma}[Static modulation near the black orbit]
\label{lem:black-modulation}
There are $\rho_0,C>0$ such that, locally near the black orbit, one can
choose continuous parameters $a,c,\theta$ with
\[
 \epsilon=e^{-i\theta}u(\cdot+a)-U_c,
 \qquad w=v(\cdot+a),
\]
satisfying
\begin{equation}\label{black:orthogonality}
 \int R_c'\epsilon_1=\int R_c'\epsilon_2
       =\int R_cQ_c\epsilon_2=0.
\end{equation}
If $\mathfrak d_0(u,v)=\delta<\rho_0$, the parameters near
$(a,c,\theta)=(0,0,0)$ satisfy
\begin{equation}\label{black:initial-modulation}
 |a|+|c|+|\theta|+D_c(\epsilon,w)\le C\delta.
\end{equation}
The same assertion holds around any translated and phase-rotated
profile, in the corresponding centered coordinates.
\end{lemma}

\begin{proof}
Let $F$ be the three left-hand sides of
\eqref{black:orthogonality}, with
$\epsilon=e^{-i\theta}u(\cdot+a)-U_c$.
The exponentially decaying weights make $F$ continuously
differentiable as a function of the weighted perturbation and the
three real parameters. Differentiating in $a$ is justified by moving
the translation onto the smooth weight, so no second derivative of
$u$ is needed. At $(u,a,c,\theta)=(U_0,0,0,0)$,
\[
 D_{(a,c,\theta)}F=
 \begin{pmatrix}
 2\sqrt2/3&0&0\\
 0&\sqrt2&0\\
 0&0&-2\sqrt2/3
 \end{pmatrix}.
\]
The implicit function theorem gives the parameters and the bound
for $|a|+|c|+|\theta|+\|\epsilon\|_{H_0^1}$.
The change from the sign convention of \cite[Proposition~2]{GS2015}
is just complex conjugation.

To verify the density bound explicitly, set
$\eta_0=|u|^2-|U_0|^2$. By
\eqref{black:distance-equivalence}, $\|\eta_0\|_2\le C\delta$.
Moreover,
\[
 \eta=\eta_0(\cdot+a)
      +|U_0(\cdot+a)|^2-|U_0|^2
      +|U_0|^2-|U_c|^2.
\]
The middle term has $L^2$ norm at most
$|a|\|\partial_x|U_0|^2\|_2$, and the last term is $O(c^2)$ in
$L^2$, by the explicit profile formula. This proves
\eqref{black:initial-modulation}, since translation preserves the
bright $H^1$ norm. Local uniqueness allows the parameters to be
continued along a continuous solution while it stays in the tube.
This is the static modulation construction of
\cite[Proposition~2]{GS2015}; no modulation differential equation
is used below.
\end{proof}

\subsection{A local real lift and its quadratic expansion}

For a modulated function $u_{\mathrm m}=U_c+\epsilon$ in the preceding
neighbourhood, $|u_{\mathrm m}|$ is bounded away from zero outside
$[-1,1]$. One elementary justification is to write
$\eta_0=|u_{\mathrm m}|^2-|U_0|^2$: smallness of
$d_0(u_{\mathrm m},U_0)$ implies $\|\eta_0\|_2\to0$, while
$\|\partial_x\eta_0\|_2$ stays bounded, by
\eqref{black:distance-equivalence} and the energy bound. Thus
$\|\eta_0\|_\infty\to0$. On the two exterior components write
$u_{\mathrm m}=|u_{\mathrm m}|e^{i\vartheta_{\rm m}}$, fixing the branches by
\[
 |\vartheta_{\rm m}-\pi|<\pi/2\quad\hbox{on }[-2,-1],
 \qquad |\vartheta_{\rm m}|<\pi/2\quad\hbox{on }[1,2].
\]
These choices are possible and unique by the local uniform
control in \eqref{black:pointwise}. Choose a smooth cutoff $\chi$,
equal to one on $[-1,1]$ and zero outside $(-2,2)$, and put
\begin{equation}\label{black:lift}
 \mathcal P(u_{\rm m})=
 \frac12\int\left\{
  \operatorname{Im}(\bar u_{\rm m}\partial_xu_{\rm m})
       -\partial_x[(1-\chi)\vartheta_{\rm m}]\right\}.
\end{equation}
The integrand is integrable, since outside $[-2,2]$ it equals
$(|u_{\rm m}|^2-1)\vartheta_{\rm m}'$. Formula
\eqref{mom:finite-core} shows that
$\mathcal P(u_{\rm m})=[P_{\rm d}](u_{\rm m})\pmod\pi$.
It is a continuous real lift on this neighbourhood, as also shown
in \cite[Proposition~4]{GS2015}.

We use the following precise consequence of
\cite[Proposition~5, equation~(29)]{GS2015}:
\begin{equation}\label{black:momentum-expansion}
 \mathcal P(U_c+\epsilon)
 =p(c)-\int R_c'\epsilon_2+\mathcal R_c(\epsilon),
 \qquad
 |\mathcal R_c(\epsilon)|
 \le A\bigl(\|\epsilon\|_{H_0^1}^2+\|\eta\|_2^2\bigr).
\end{equation}
Its hypotheses are exactly that the function belongs to a
sufficiently small $\mathcal V_0$ and that the parameters are those
of the static modulation above. More explicitly, conjugating the
convention in that paper gives
$\mathcal P_{\rm here}(u)=\pi-\mathcal P_{\rm GS}(\bar u)$;
the linear term in \eqref{black:momentum-expansion} has the displayed
minus sign. Direct evaluation on our profiles gives
\begin{equation}\label{black:profile-momentum}
 p(c)=\frac\pi2+\arctan\frac{c}{\sqrt{2-c^2}}
                +\frac c2\sqrt{2-c^2},
 \qquad p'(c)=\sqrt{2-c^2}.
\end{equation}
Adding $\frac12\int\operatorname{Im}(\bar w w_x)$ to
\eqref{black:lift} defines a continuous real lift $\mathscr P$ of the
total momentum. By \eqref{black:orthogonality},
\eqref{black:momentum-expansion}, and Cauchy--Schwarz,
\begin{equation}\label{black:momentum-error}
 |\mathscr P-p(c)|\le A_*D_c(\epsilon,w)^2.
\end{equation}
Along a continuous solution in this tube, Proposition~\ref{prop:momentum}
implies conservation of $\mathscr P$: a continuous real function
whose class modulo $\pi$ is constant must itself be constant.

\subsection{The conserved-functional argument}

Let $\delta=d_0^{\mathrm{orb}}(\psi_0,\phi_0)$ be sufficiently small.
For $\delta>0$, choose a fixed initial translation and phase rotation
so that $\mathfrak d_0(\psi_0,\phi_0)<2\delta$.
If $\delta=0$, apply the argument with any positive upper bound for
$\delta$ and then let that bound tend to zero.
Theorem~\ref{thm:gwp} supplies a global solution.
First apply Lemma~\ref{lem:black-modulation}, and continue its
parameters on the connected interval containing zero on which the
modulated solution remains in a fixed small tube. Write
$D(t)=D_{c(t)}(\epsilon(t),w(t))$. By
\eqref{black:initial-modulation},
\begin{equation}\label{black:initial-size}
 D(0)+|c(0)|\le C\delta.
\end{equation}

Set $J=H+2M$ and
$h(c)=J(U_c,0)=\frac13(2-c^2)^{3/2}$.
This is a conserved functional, whereas $h(c(t))$ is not constant.
Proposition~\ref{prop:augmented-coercivity} gives fixed constants
$\kappa,B>0$ such that
\[
 \kappa D(t)^2\le J(\psi(t),\phi(t))-h(c(t)),
 \qquad
 |J(\psi_0,\phi_0)-h(c(0))|\le B D(0)^2.
\]
Since $h(c(0))\le h(0)$, conservation and
$0\le h(0)-h(c)\le Cc^2$ imply
\begin{equation}\label{black:energy-bootstrap}
 D(t)^2\le C_1\delta^2+C_2c(t)^2.
\end{equation}
This estimate retains the change in the energy of the reference profile.

Fix the tube so that $|c(t)|<s<1$. Using conservation of
$\mathscr P$, \eqref{black:momentum-error}, and
$p'(c)\ge\sqrt{2-s^2}$ gives
\[
 |c(t)-c(0)|
 \le\frac{A_*}{\sqrt{2-s^2}}\bigl(D(t)^2+D(0)^2\bigr).
\]
Thus, by \eqref{black:initial-size},
\begin{equation*}
 |c(t)|\le C_3\delta+C_4D(t)^2.
\end{equation*}
Substitution into \eqref{black:energy-bootstrap} yields
$D(t)^2\le C_5\delta^2+C_6D(t)^4$.
Choose the tube's perturbation radius $r$ so small that
$C_6r^2\le1/2$. As long as $D(t)<r$,
\begin{equation}\label{black:closed-bootstrap}
 D(t)\le\sqrt{2C_5}\,\delta,
 \qquad |c(t)|\le C_7\delta.
\end{equation}
All tube radii and constants are fixed before choosing $\delta$.

To return to the fixed black profile, the explicit formula gives
\[
 \|U_c-U_0\|_{H_0^1}\le C|c|,
 \qquad \||U_c|^2-|U_0|^2\|_2\le Cc^2.
\]
The triangle inequality for the derivative and weighted terms,
followed by \eqref{black:distance-equivalence} centered at $U_0$,
therefore gives
\begin{equation}\label{black:fixed-orbit}
 \mathfrak d_0(e^{-i\theta(t)}\psi(t,\cdot+a(t)),
                         \phi(t,\cdot+a(t)))
 \le C\bigl(D(t)+|c(t)|\bigr)\le C\delta.
\end{equation}
In particular, the dark component remains in the precise
$\mathcal V_0$ required for \eqref{black:momentum-expansion}.

For completeness, we check that continuing the modulation introduces
no finite-time obstruction. If its interval had a finite endpoint $T$,
global-flow continuity would give convergence of $\psi_x(t)$ in
$L^2$ as $t\to T$. When $r$ is small, the decomposition
\[
 e^{-i\theta(t)}\psi_x(t,\cdot+a(t))=U_{c(t)}'+\epsilon_x(t)
\]
forces a fixed positive amount of derivative mass into a uniformly
bounded neighbourhood of the centre $a(t)$. Tightness of
$\psi_x(t)$ in $L^2$ then keeps $a(t)$ bounded as $t\to T$.
After passing to a subsequence, the parameters $a(t),c(t)$ and
$e^{i\theta(t)}$ converge. Continuity gives the corresponding
decomposition at $T$, still strictly inside the tube by
\eqref{black:closed-bootstrap} when $\delta$ is small enough.
The implicit function theorem at this limiting decomposition, using
a local angular coordinate for $e^{i\theta}$, supplies a unique nearby
parameter branch for times near $T$, including times beyond $T$.
Choose a smaller parameter neighbourhood whose closure lies in the
uniqueness neighbourhood. By continuity, the new branch stays
strictly inside the smaller neighbourhood for times sufficiently
close to $T$. The original continuous branch enters that smaller
neighbourhood along the convergent subsequence. It cannot subsequently
leave: at a first boundary crossing it would still lie in the
uniqueness neighbourhood and hence coincide with the new branch,
which lies strictly inside. Thus the branches agree on a terminal
time interval, extending the original branch past $T$, a contradiction.
The same reasoning applies backwards in time.
Consequently \eqref{black:fixed-orbit} holds for all $t\in\mathbb R$,
which proves the zero-speed part of Theorem~\ref{thm:stability}.

The use of $J=H+2M$ is essential for removing the unaugmented bright
energy threshold from this argument: its bright quadratic operator
is $-\partial_x^2+1+\alpha|U_c|^2$, which is strictly positive.
No evenness or oddness restriction has been imposed on the data.

\section{Explicit dark--bright waves and an open stability problem}
\label{sec:dark-bright-open}

The open problem in Remark~\ref{rem:dark-bright-open} can be posed for
an explicit family. Hyperbolic-function dark--bright profiles in
coupled Schr\"odinger equations are studied, for example, in
\cite{Yan2015,MaoZhao2022}. We record the family in the normalization
and phase convention of \eqref{eq:system}, and verify the formulas
directly. No stability conclusion for this family is asserted below.

\subsection{The explicit family and its admissible parameters}

In this section we impose the additional assumptions
\[
 \alpha>1,\qquad \beta>\alpha^2.
\]
They ensure the admissibility of the explicit family below; they are
not needed for stability of the GP branch in Theorem~\ref{thm:stability}.
Set
\begin{equation}\label{db:parameters}
\begin{aligned}
 c_{\rm DB}^2&=\frac{2(\beta-\alpha^2)}{\beta-\alpha},
 &b=b(c)&=\frac12\sqrt{c_{\rm DB}^2-c^2},
 \qquad |c|<c_{\rm DB},\\
 A_1&=\frac{2b}{c_{\rm DB}},
 &B_1&=-\frac{c}{c_{\rm DB}},\\
 A_2&=b\sqrt{\frac{2(\alpha-1)}{\beta-\alpha^2}},
 &\nu_*&=\alpha-1-\frac{c_{\rm DB}^2}{4}.
\end{aligned}
\end{equation}
Here $c_{\rm DB}$ is the positive square root. The assumptions imply
$\beta>\alpha^2>\alpha$, and hence $0<c_{\rm DB}<\sqrt2$.
For every admissible speed, $b,A_1,A_2$ are strictly positive.
Define
\begin{equation}\label{db:profiles}
\begin{aligned}
 U_{d,c}(x)&=A_1\tanh(bx)+iB_1,\\
 W_{{\rm DB},c}(x)&=A_2\sech(bx),\qquad
 V_{b,c}(x)=e^{-icx/2}W_{{\rm DB},c}(x).
\end{aligned}
\end{equation}
Then the exact solution is
\begin{equation}\label{db:solution}
 \psi(t,x)=U_{d,c}(x-ct),\qquad
 \phi(t,x)=e^{i\nu_*t}V_{b,c}(x-ct).
\end{equation}
In particular, $V_{b,c}\not\equiv0$, $|U_{d,c}(x)|\to1$ and
$V_{b,c}(x)\to0$ as $|x|\to\infty$.

To check the formulas, a travelling--standing wave
$(U(x-ct),e^{i\nu t}V(x-ct))$ must satisfy
\begin{equation}\label{db:profile-equations}
\begin{aligned}
 U''+icU'+(1-|U|^2-\alpha|V|^2)U&=0,\\
 V''+icV'+(1+\nu-\alpha|U|^2-\beta|V|^2)V&=0.
\end{aligned}
\end{equation}
Substitution of \eqref{db:profiles} reduces these equations to
\begin{equation*}
\begin{aligned}
 A_1^2+B_1^2&=1,&
 A_1^2-\alpha A_2^2&=2b^2,\\
 \alpha A_1^2-\beta A_2^2&=2b^2,&
 2bB_1+cA_1&=0,\\
 b^2+1+\nu+c^2/4-\alpha&=0.
\end{aligned}
\end{equation*}
Indeed, the dark equation uses
$U'=A_1b\sech^2(bx)$ and
$U''=-2A_1b^2\tanh(bx)\sech^2(bx)$.
For the bright equation the essential cancellation is
\[
\begin{aligned}
 V''+icV'&=e^{-icx/2}\bigl(W_{{\rm DB},c}''+c^2W_{{\rm DB},c}/4\bigr),\\
 W_{{\rm DB},c}''&=b^2W_{{\rm DB},c}
                  -2b^2\sech^2(bx)W_{{\rm DB},c}.
\end{aligned}
\]
The two amplitude equations give
\[
 A_1^2=\frac{2b^2(\beta-\alpha)}{\beta-\alpha^2},
 \qquad A_2^2=\frac{2b^2(\alpha-1)}{\beta-\alpha^2},
 \qquad B_1=-\frac{cA_1}{2b}.
\]
Using $A_1^2+B_1^2=1$ now gives \eqref{db:parameters}, including
$\nu=\nu_*$. Moreover,
\[
 1-|U_{d,c}|^2=A_1^2\sech^2(bx),
\]
so $(U_{d,c},V_{b,c})\in\mathcal X$.

The spatial phase in $V_{b,c}$ cannot be dropped when $c\ne0$:
otherwise the term $icW_{{\rm DB},c}'$ in the bright profile equation would
remain. The frequency of the real-envelope equation is
\begin{equation*}
\begin{aligned}
 \mu_c&=\nu_*+\frac{c^2}{4}=\alpha-1-b(c)^2,\\
 -\mu_cW_{{\rm DB},c}&=W_{{\rm DB},c}''
   +(1-\alpha|U_{d,c}|^2-\beta W_{{\rm DB},c}^2)W_{{\rm DB},c}.
\end{aligned}
\end{equation*}
Thus $\nu_*$, the frequency multiplying $t$ in
\eqref{db:solution}, is independent of $c$, whereas $\mu_c$ is not.
Equivalently, in laboratory coordinates the bright component is
$e^{i(\nu_*+c^2/2)t-icx/2}W_{{\rm DB},c}(x-ct)$.

At $c=0$ these formulas give
\[
 U_{d,0}(x)=\tanh(c_{\rm DB}x/2),\qquad
 V_{b,0}(x)=\sqrt{\frac{\alpha-1}{\beta-\alpha}}\,
                         \sech(c_{\rm DB}x/2).
\]
This is a standing dark--bright wave, with the temporal bright phase
$e^{i\nu_*t}$ still present. As $c\to\pm c_{\rm DB}$, the width
parameter $b$ and both amplitudes $A_1,A_2$ tend to zero:
$U_{d,c}\to\mp i$ uniformly and $V_{b,c}\to0$ in $H^1$.
The endpoints are therefore constant backgrounds, not nontrivial
localized members of the family.

\subsection{The relevant conserved action}

Write $P_{\rm b}(w)=\frac12\int\operatorname{Im}(\bar w w_x)$.
The spatial phase has the following exact effects:
\begin{equation*}
\begin{aligned}
 M(V_{b,c})&=\frac{A_2^2}{b},\qquad
 P_{\rm b}(V_{b,c})=-\frac c2M(V_{b,c}),\\
 H(U_{d,c},V_{b,c})&=H(U_{d,c},W_{{\rm DB},c})
                    +\frac{c^2}{4}M(V_{b,c}).
\end{aligned}
\end{equation*}
Thus a calculation with the real envelope alone would omit
the phase-gradient contribution to the physical energy.

For a fixed admissible $c$, the appropriate travelling-wave action is
\begin{equation}\label{db:action}
 \mathcal S_c=H+cP-\nu_*M,
\end{equation}
where $P$ is a local real lift of the \emph{total} untwisted momentum.
For $c\ne0$, the explicit profile satisfies
$\inf_{\R}|U_{d,c}|=|c|/c_{\rm DB}>0$, so the real-valued momentum
\eqref{mom:nonvanishing} can be used in a nonvanishing neighbourhood
of this fixed profile. At $c=0$, a local lift can instead be chosen
by fixing the exterior phase branches in \eqref{mom:finite-core}.
The action itself then reduces to $\mathcal S_0=H-\nu_*M$ and needs
no momentum lift; such a lift is still required when using real
momentum constraints or squared momentum penalties at zero speed.
The chosen lift is conserved along a solution while the solution remains
in that chart, by Theorem~\ref{thm:gwp} and Proposition~\ref{prop:momentum}.
The first variation of $P$ on affine $H^1$ charts is
$DP(U,V)=(-iU',-iV')$. By \eqref{db:profile-equations},
\[
\begin{aligned}
 DH(U_{d,c},V_{b,c})
   &=\bigl(icU_{d,c}',\,icV_{b,c}'+\nu_*V_{b,c}\bigr),\\
 D\mathcal S_c(U_{d,c},V_{b,c})&=0.
\end{aligned}
\]
These identities specify the first variations against $H^1$
perturbations, rather than assuming a linear structure on all of
$\mathcal X$.

It is important here to use the total momentum. For smooth solutions,
the two components exchange momentum according to
\[
 \frac{d}{dt}P_{\rm b}(\phi(t))
   =\frac{\alpha}{2}\int\partial_x|\psi|^2\,|\phi|^2.
\]
Consequently $H+cP_{\rm b}-\nu_*M$ is generally not conserved when
$c\ne0$, even though it can be convenient for static calculations.
Furthermore, replacing $\mathcal S_c$ by $\mathcal S_c+\kappa M$
with $\kappa\ne0$ destroys stationarity at this fixed profile:
$DM(V_{b,c})=V_{b,c}\ne0$. The bright-mass correction used for the
GP branch cannot therefore be transferred without modification.
Conserved squared penalties
$K_P(P-P_*)^2+K_M(M-M_*)^2$, with $P_*,M_*$ equal to the profile
values, do preserve stationarity. They are possible tools for a
constrained argument, but adding them does not by itself prove
coercivity of the coupled Hessian.

One can see the remaining quadratic-form issue explicitly.
For $\epsilon,z\in H^1(\R;\C)$ write
\[
\begin{gathered}
 U=U_{d,c},\quad W=W_{{\rm DB},c},\quad
 \omega=e^{-icx/2}z,\quad q_{{\rm DB},c}=2b^2\sech^2(bx),\\
 \eta=2\operatorname{Re}(\bar U\epsilon)+|\epsilon|^2,\quad
 \zeta=2W\operatorname{Re}z+|z|^2.
\end{gathered}
\]
Expanding \eqref{db:action}, with a consistent local real momentum
lift, gives
\begin{equation}\label{db:action-expansion}
\begin{aligned}
 &\mathcal S_c(U+\epsilon,V_{b,c}+\omega)
                   -\mathcal S_c(U,V_{b,c})\\
 &\quad=\frac12\int\bigl(
       |\epsilon_x|^2-q_{{\rm DB},c}|\epsilon|^2
          +c\operatorname{Im}(\bar\epsilon\epsilon_x)\bigr)\\
 &\qquad\quad+\frac12\int\bigl(|z_x|^2+(b^2-q_{{\rm DB},c})|z|^2\bigr)\\
 &\qquad\quad+\frac14\int(\eta^2+2\alpha\eta\zeta+\beta\zeta^2).
\end{aligned}
\end{equation}
The identity follows from \eqref{mom:affine} and the square completion
\[
 \frac12\int|(e^{-icx/2}f)_x|^2+cP_{\rm b}(e^{-icx/2}f)
       =\frac12\int|f_x|^2-\frac{c^2}{8}\int|f|^2.
\]
In particular, the dark momentum contributes the indefinite term
$\frac c2\int\operatorname{Im}(\bar\epsilon\epsilon_x)$, and the
quadratic part contains the mixed term
$2\alpha\int W\operatorname{Re}(\bar U\epsilon)\operatorname{Re}z$.
Thus the quadratic form does not split into the scalar GP form and
an independent positive bright form. Identity
\eqref{db:action-expansion} is stated on affine $H^1$ charts only;
a stability proof for the full orbital energy topology must also
justify the required coordinate and coercivity estimates there.

The neutral bright-phase direction can also be read off explicitly.
On real $L^2(\R)$ with domain $H^2(\R)$, the bright scalar operator
factorizes as
\[
\begin{aligned}
 L_{{\rm b},c}
 &= -\partial_x^2+b^2-q_{{\rm DB},c}\\
 &=\bigl(-\partial_x+b\tanh(bx)\bigr)
   \bigl(\partial_x+b\tanh(bx)\bigr).
\end{aligned}
\]
Hence $L_{{\rm b},c}\ge0$ and
$\ker L_{{\rm b},c}=\operatorname{span}\{W\}$, since the first-order
equation for its kernel gives multiples of $\sech(bx)$.
The full action Hessian annihilates $(\epsilon,z)=(0,iW)$:
the mixed and remaining quadratic terms vanish because
$\operatorname{Re}z=0$. In the original variables this is the tangent
$(0,iV_{b,c})$ to the independent bright-phase orbit. Thus coercivity
must account for this neutral direction, which explains the second
phase parameter in \eqref{db:distance}.
This is a statement about the Hessian, not the full nonlinear
increment: for real $\tau$, \eqref{db:action-expansion} gives
\[
 \mathcal S_c(U,(1+i\tau)V_{b,c})-\mathcal S_c(U,V_{b,c})
   =\frac{\beta\tau^4}{4}\int W^4,
\]
whereas the exact phase perturbation $z=(e^{i\tau}-1)W$ leaves
the action unchanged.

\subsection{A precise orbital-stability question}

Fix $R_{\rm loc}>0$ and define the centered distance to the explicit
profile by
\begin{equation}\label{db:distance}
\begin{aligned}
 \rho_{R_{\rm loc},c}(\psi,\phi)
 ={}&\|\psi-U_{d,c}\|_{L^\infty([-R_{\rm loc},R_{\rm loc}])}
      +\|\psi_x-U_{d,c}'\|_2\\
    &+\big\||\psi|^2-|U_{d,c}|^2\big\|_2
      +\|\phi-V_{b,c}\|_{H^1},\\
 \rho_{R_{\rm loc},c}^{\rm orb}(\psi,\phi)
 ={}&\inf_{a,\theta_1,\theta_2\in\R}
   \rho_{R_{\rm loc},c}\bigl(e^{-i\theta_1}\psi(\cdot+a),
                  e^{-i\theta_2}\phi(\cdot+a)\bigr).
\end{aligned}
\end{equation}
The independent bright phase accounts for the factor $e^{i\nu_*t}$
in the exact solution. The speed in \eqref{db:distance} is fixed.

The concrete open problem is to determine for which
$\alpha>1$, $\beta>\alpha^2$ and $|c|<c_{\rm DB}$ the following
orbital-stability assertion holds: for every $\varepsilon>0$ there
exists $\delta>0$ such that, for all $(\psi_0,\phi_0)\in\mathcal X$,
\[
 \rho_{R_{\rm loc},c}^{\rm orb}(\psi_0,\phi_0)<\delta
 \quad\Longrightarrow\quad
 \sup_{t\in\R}\rho_{R_{\rm loc},c}^{\rm orb}(\psi(t),\phi(t))<\varepsilon.
\]
The global flow is already supplied by Theorem~\ref{thm:gwp}.
The unresolved part is stability under unrestricted two-component
perturbations, not the existence of these explicit waves. Along a
variational route one must establish suitable constrained coercivity
of the coupled action and control the neutral symmetry directions.
The explicit family has the single parameter $c$ and the fixed
frequency $\nu_*$; by itself it does not supply a two-parameter
speed--frequency family for a general two-constraint slope criterion.
Comparing the energies of the GP and dark--bright branches alone
does not resolve this question. We leave the rigorous analysis open
here; however, we present numerical simulations
(Figures~\ref{fig:gaussian_perturbation_1}--\ref{fig:dark_bright_alpha4_beta20_c0})
suggesting that the dark (black)--bright soliton is not only orbitally
stable but also asymptotically stable in some sense.

\paragraph{AI assistance statement.}
The authors used AI tools to assist with manuscript editing and checking mathematical arguments they developed. All proofs were independently verified by
the authors, who take full responsibility for the mathematical correctness and final content of the
manuscript.

\begin{figure}
    \centering
    {\large\textbf{Dark--bright soliton: numerical vs.\ exact profiles, $\alpha = 4$, $\beta = 20$, $c = 0.5$}}\\[0.5em]
    \includegraphics[width=\textwidth]{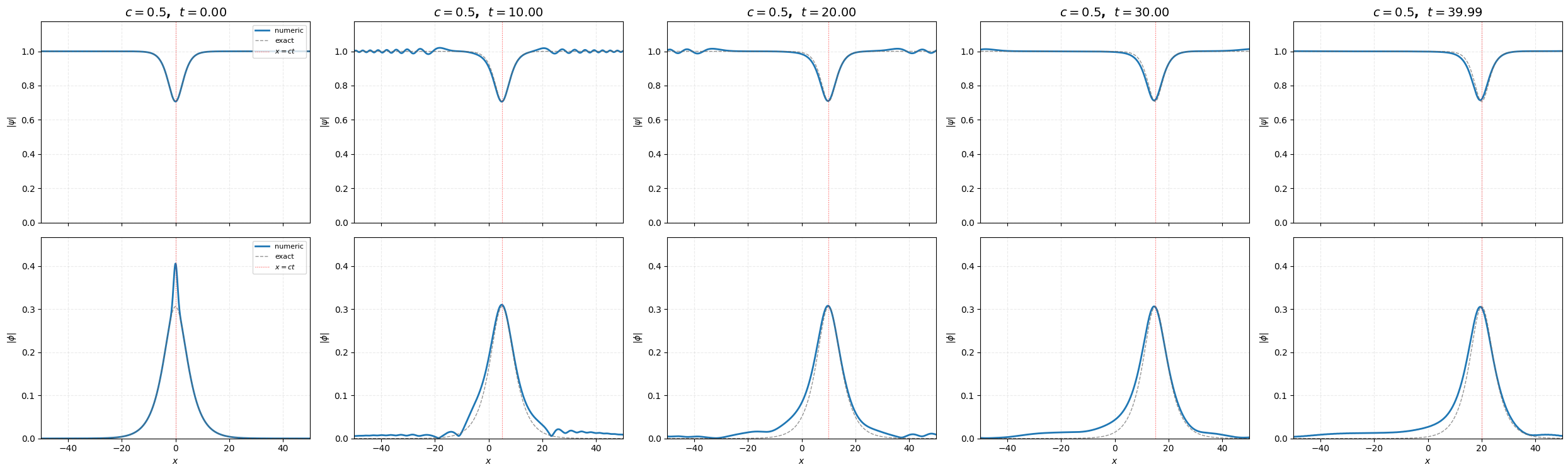}
    \caption{Numerical solution of the system for $\alpha = 4$, $\beta = 20$ and $c = 0.5$ with a Gaussian perturbation of the initial data, $u_i \exp\!\left(-x^{2}/\sigma^{2}\right)$, where $\sigma = 1$, $u_i = 0.5$ for $\psi$ and $u_i = 0.05$ for $\phi$. Solid: numerical solution; dashed: exact profile. Top: $|\psi|$; bottom: $|\phi|$.}
    \label{fig:gaussian_perturbation_1}
\end{figure}

\begin{figure}
    \centering
    {\large\textbf{Dark--bright soliton, $\alpha = 4$, $\beta = 20$, $c = 0.5$}}\\[0.5em]
    \includegraphics[width=\textwidth]{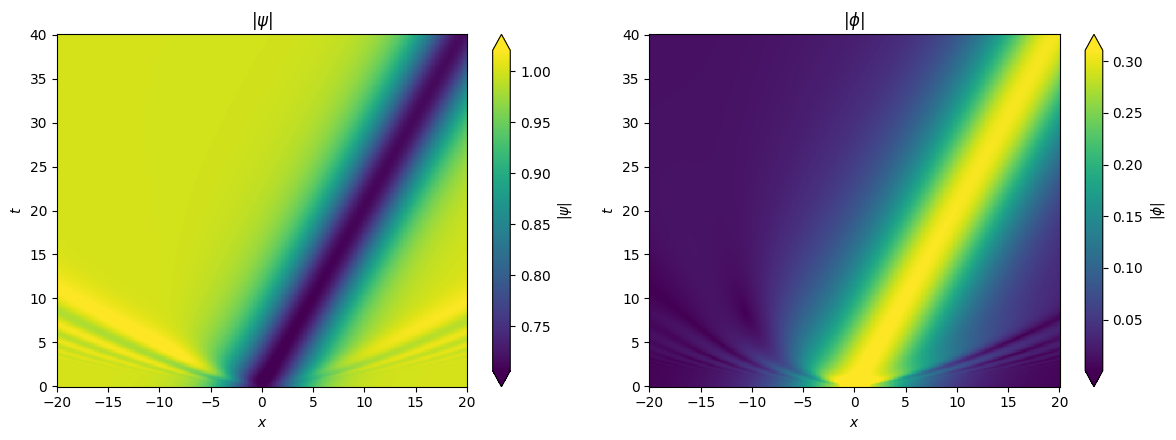}
    \caption{Numerical solution of the system for $\alpha = 4$, $\beta = 20$ and $c = 0.5$ with a Gaussian perturbation of the initial data, $u_i \exp\!\left(-x^{2}/\sigma^{2}\right)$, where $\sigma = 1$, $u_i = 0.5$ for $\psi$ and $u_i = 0.05$ for $\phi$. Left: $|\psi|$; right: $|\phi|$, as functions of $x$ and $t$.}
    \label{fig:dark_bright_alpha4_beta20_c05}
\end{figure}

\begin{figure}
    \centering
     {\large\textbf{Black--bright soliton: numerical vs.\ exact profiles, $\alpha = 4$, $\beta = 20$, $c = 0$}}\\[0.5em]
    \includegraphics[width=\textwidth]{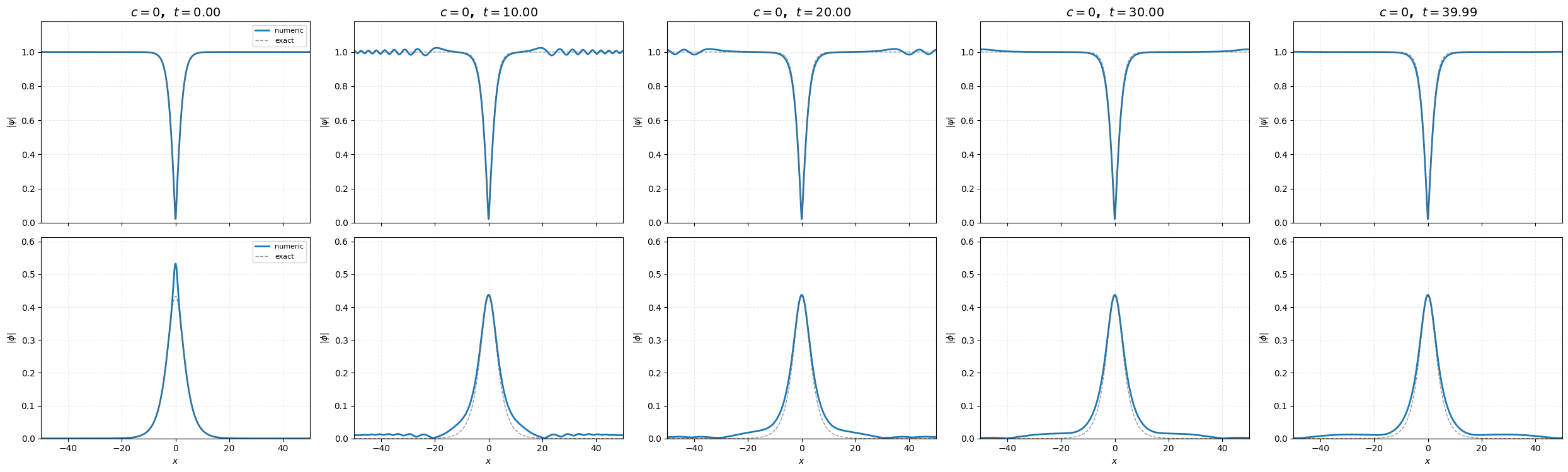}
    \caption{Numerical solution of the system for $\alpha = 4$, $\beta = 20$ and $c = 0$ with a Gaussian perturbation of the initial data, $u_i \exp\!\left(-x^{2}/\sigma^{2}\right)$, where $\sigma = 1$, $u_i = 0.5$ for $\psi$ and $u_i = 0.05$ for $\phi$. Solid: numerical solution; dashed: exact profile. Top: $|\psi|$; bottom: $|\phi|$.}
    \label{fig:gaussian_perturbation_2}

\end{figure}

\begin{figure}
    \centering
      {\large\textbf{Black--bright soliton, $\alpha = 4$, $\beta = 20$, $c = 0$}}\\[0.5em]
    \includegraphics[width=\textwidth]{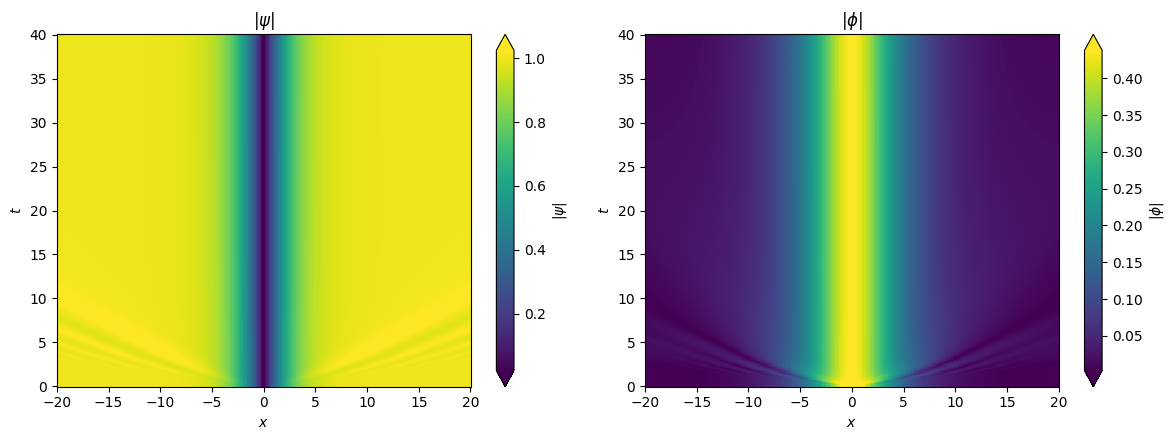}
    \caption{Numerical solution of the system for $\alpha = 4$, $\beta = 20$ and $c = 0$ with a Gaussian perturbation of the initial data, $u_i \exp\!\left(-x^{2}/\sigma^{2}\right)$, where $\sigma = 1$, $u_i = 0.5$ for $\psi$ and $u_i = 0.05$ for $\phi$. Left: $|\psi|$; right: $|\phi|$, as functions of $x$ and $t$.}
    \label{fig:dark_bright_alpha4_beta20_c0}
\end{figure}

\end{document}